\documentclass[
12pt]{amsart}
\usepackage{amssymb,amsmath, amsthm,latexsym}
\usepackage{a4wide}

\usepackage{hyperref}
\usepackage{color}

\theoremstyle{plain}

\newtheorem{lemma}{Lemma}[section]
\newtheorem{theorem}[lemma]{Theorem}
\newtheorem{proposition}[lemma]{Proposition}
\newtheorem{corollary}[lemma]{Corollary}

\newtheorem{remark}[lemma]{Remark}

\newtheorem{definition}[lemma]{Definition}

\parskip=\bigskipamount

\font\rm=cmr12

\def\pc{\pi_{cusp}}
\def\Z{\mathbb Z}

\def\d{\delta}
\def\r{\rtimes}
\def\t{\times}
\def\o{\otimes}
\def\e{\epsilon}

\def\jp{\text{\rm Jord}(\pi)}
\def\jrp{\text{\rm Jord}_\rho(\pi)}

\def\h{\hookrightarrow}

\def\a{\alpha}
\def\b{\beta}

\def\D{\Delta}

\def\s{\sigma}

\def\ep{\epsilon_\pi}

\def\av{ |\det  |_{_{_{\!F}}}}

\title[On tempered representations]
{On tempered and square integrable representations of classical $p$-adic groups}

\mark{Marko Tadi\'c}

\author{Marko Tadi\'c}

\address{Department of Mathematics, University of Zagreb
\\
Bijeni\v{c}ka 30, 10000 Zagreb,
 Croatia\\
Email: \tt tadic{\char'100}math.hr}


\thanks{
The   
author was partly supported by 
Croatian Ministry of Science, Education and Sports grant
{\#}037-0372794-2804.}
\date{\today}

\begin{document}

\begin{abstract} This paper has two aims. The first is to give a description of irreducible tempered representations of classical $p$-adic groups which follows naturally the classification of irreducible square integrable representations modulo cuspidal data obtained in \cite{Moe-Ex} and \cite{Moe-T}.  The second aim of the paper is to give description of an invariant (partially defined function)  of   irreducible square integrable representation of a classical $p$-adic group (defined by C. M\oe glin using embeddings) in terms of subquotients of Jacquet modules. 
As an application, we describe behavior of partially defined function in one  construction of square integrable representations of a bigger group from such  representations of a smaller group (which is related to deformation of Jordan blocks of representations).
\end{abstract}

\maketitle

\setcounter{tocdepth}{1}


\tableofcontents

\section{Introduction}\label{intro}

In this paper we
 shall fix a non-archimedean local field $F$ and consider  irreducible tempered and square integrable representations of classical groups over $F$.
 
 First we shall describe parameterization of tempered representations obtained in this paper. These representations are important for a number of reasons (Plancherel measure, non-unitary dual, orbital integrals etc.).

At the beginning, we shall  recall a fundamental result of D. Goldberg  on tempered representations  of a classical group $G$ over $F$ (\cite{G}). Levi factor $M$ of a proper parabolic subgroup $P$ of $G$ is isomorphic to a direct product $GL(n_1,F)\t \dots \t GL(n_k,F)\t G'$, where $G'$ is a classical group from the same series as $G$, whose split rank is smaller then the split rank of $G$ (see section \ref{notation} for more details). 

\begin{theorem} {\rm (D. Goldberg)} Take irreducible square integrable (modulo center) representations $\d_i$ of $GL(n_i,F)$, $i=1,\dots,k$, and  an irreducible square integrable representation  $\pi$ of $G'$. Denote by $l$ the number of non-isomorphic $\d_i$'s such that  the parabolically induced representation 
$$
\text{Ind}^{G_i}(\d_i\o\pi)
$$
of the appropriate classical group $G_i$ reduces. Then  the parabolically induced representation 
\begin{equation}
\label{sub}
\text{Ind}_P^G(\d_1\o\dots\o\d_k\o\pi)
\end{equation}
is a multiplicity one representation of length $2^l$. Further, if $\tau$ is (equivalent to) an irreducible subrepresentation of some representation \eqref{sub} as above, then $\tau$ determines (equivalence class of) $\pi$, and it determines (equivalence classes of) $\d_1,\dots,\d_k$ up to a permutation and taking contragredients\footnote{In the case of unitary groups, one needs to consider Hermitian contragredients}.
\end{theorem}

This result reduces the problem of description of irreducible tempered representations to square integrable representations and tempered reducibilities in the generalized rank one case. The rank one reducibilities are part of the classification of square integrable representations of classical groups modulo cuspidal data in \cite{Moe-Ex} and \cite{Moe-T} (we shall say later more regarding this). The theory of $R$-groups gives a parameterization of irreducible pieces of $\text{Ind}_P^G(\d_1\o\dots\o\d_k\o\s)$ in terms of characters of $R$-groups. In this paper we shall give description of irreducible pieces by parameters coming from the parameters of square integrable representations  of the classification   in \cite{Moe-Ex} and \cite{Moe-T}.

First observe that for parameterizing irreducible pieces of $\text{Ind}_P^G(\d_1\o\dots\o\d_k\o\s)$, it is enough to know to parameterize them  in the case $l=k$ (further tempered parabolical induction is irreducible). Therefore, we shall assume $l=k$ in what follows. For this case,  we have the following simple reduction to the generalized rank one case.

Each representation $
\text{Ind}^{G_i}(\d_i\o\pi)
$ splits into two irreducible non-isomorphic representations. Denote these pieces by $\pi_{\d_i}$ and $\pi_{-\d_i}$, i.e.
\begin{equation}
\label{Eq-pm}
\text{Ind}^{G_i}(\d_i\o\pi)= \pi_{\d_i} \oplus \pi_{-\d_i}
\end{equation}
(later on, we shall come to the problem of parameterizing irreducible pieces of $\text{Ind}^{G_i}(\d_i\o\pi)$). 
  Let $j_1,\dots,j_k\in\{\pm\}$. Then there exists a unique irreducible subrepresentation $\tau$ of $
\text{Ind}_P^G(\d_1\o\dots\o\d_k\o\pi)
$
such that $\tau$ is a subrepresentation of 
$$
\text{Ind}_P^G(\d_1\o\dots\d_{i-1}\o\d_{i+1}\o\dots\d_k\o\pi_{j_i\d_i}),
$$
for each $i=1,\dots,k$. We denote such $\tau$   by
$$
\pi_{j_1\d_1,\dots,j_k\d_k}.
$$

Therefore, to get a parameterization of irreducible tempered representations of classical groups, 
it remains  to determine in \eqref{Eq-pm} which irreducible subrepresentation 
  will be denote  by $\pi_{\d_i}$  (we have two choices; the other irreducible subrepresentation is then denoted by $\pi_{-\d_i}$).  To describe which subrepresentation will be denoted by $\pi_{\d_i}$, we shall briefly recall  the notion of Jordan blocks attached to an irreducible square integrable representation of a classical group (Jordan blocks $Jord(\pi)$ attached to an irreducible square integrable representation $\pi$ of a classical group is one of three invariants which classify 
  irreducible square integrable representations  of a classical groups modulo cuspidal data, and a natural assumption).

Before we recall the definition of Jordan blocks, we shall   recall  some notation for general linear groups.
 Let $\rho$ be an irreducible cuspidal representation of $GL(p,F)$ and let $n$ be a positive integer (we  consider only smooth representations in this paper). Let 
 $$
 \big[\, \rho,\av^n\rho\,\big]\ :\,=\ \{\rho,\av \rho,
\av^2\rho,\dots,\av^n\rho\}
 $$
 ($ |\  \, |_{_{_{\!F}}}$ denotes the normalized absolute value on $F$).
The parabolically
induced representation
$$
\text{Ind}^{GL(np,F)}(\av^n\rho\otimes \av^{n-1}\rho\otimes \dots\otimes \av\rho\otimes \rho),
$$
induced from the appropriate  parabolic subgroup which is standard with respect to the minimal parabolic subgroup of all upper triangular matrices in the group,  contains a unique irreducible subrepresentation. This subrepresentation is  denoted by
$$
\delta( \big[\, \rho,\av^n\rho\,\big])
$$
(the parabolic induction that we consider in this paper is normalized). Then the representation $\delta( \big[\, \rho,\av^n\rho\,\big])$  is an essentially square integrable representation. Denote
$$
\delta( \rho,n):=\delta( \big[\, \av^{-\frac{n-1}2}
\rho,\av^{\frac{n-1}2}\rho\,\big]).
$$

For simplicity,  in the introduction we shall only deal  with symplectic and split special odd-orthogonal groups (for the definition of these groups see  section \ref{notation}). We shall fix one of these series of groups, and denote the group of split rank $n$ in the series by $S_n$.

Let $\pi$ be an irreducible square integrable representation  of $S_q$. In what follows, we shall assume that  a natural hypothesis, called basic assumption, holds (this is (BA) in section \ref{notation}). Fix an irreducible selfdual representation $\rho$ of a general linear group (selfdual means that the contragredient representation $\tilde \rho$ of $\rho$ is isomorphic to $\rho$).
Consider    representations
\begin{equation}\label{Eq-ind}
\text{Ind}^{S_{np+q}}(\delta(  \rho,n)\otimes \pi),
\end{equation}
parabolically induced from  appropriate parabolic subgroups.
Then for one parity of $n$ in $\mathbb Z_{>0}$, the corresponding representations \eqref{Eq-ind} are always irreducible, while for the other parity we have always reducibility, except for finitely many $n$. All the exceptions $n$  are denote by
$$
\jrp.
$$
Then the Jordan blocks of $\pi $ are defined by
$$
\jp=\bigcup_\rho \ \{\rho\}\t\jrp,
$$
when $\rho$ runs over all equivalence classes of  irreducible selfdual cuspidal representations of general linear groups (see section  \ref{notation} for more details).

Let us recall that Jordan blocks are one of the invariants that C. Moeglin has attached in \cite{Moe-Ex} to
 an irreducible square integrable representation $\pi$ of a classical group over $F$.
 To such $\pi$,  she has also attached invarinat
 $$
 \ep,
 $$  called the partially defined function of $\pi$ , and  an irreducible cuspidal representation 
 $$
 \pc
 $$
of a classical group,   called the partial cuspidal support of $\pi$.
  The importance of these invarinats come from the fact that  triples
$$
(\jp,\ep,\pc),
$$ 
 classify irreducible square integrable representations  of  classical groups modulo cuspidal data (and a natural assumption; see \cite{Moe-T}). The definition of the partial cuspidal support will be recalled later  in the introduction.

Since the above invariants classify irreducible square integrable representations, it is important  to know if their definition is canonical. This is the case for $\jp$ and $\pc$. The definition of $\ep$ in \cite{Moe-Ex} is given in terms of embeddings in some cases, and in terms of normalized standard integral intertwining operators in the other cases\footnote{One can find in \cite{T-inv} the definition which does not use the  normalized standard intertwining operators}. The part of the definition  given by embeddings is also canonical. Only the part relaying on normalized standard  intertwining operators  is not canonical, since it depends on the choice of normalization of the operators that one uses in  the definition.
This non-canonical case of the definition occurs precisely when
$Jord_\rho(\pi)$ is a non-empty subset  of   odd integers, while  $Jord_\rho(\pc)=\emptyset$.
The last condition is
equivalent to the fact that 
\begin{equation}
\label{norm-red}
\text{Ind}^{S_{p+q'}}(\rho\otimes\pi_{cusp})
\end{equation}
reduces.

 Let us briefly explain how one can fix a normalization as above. We suppose that \eqref{norm-red} reduces (as was the case above).
Then \eqref{norm-red} reduces into two
nonequivalent irreducible pieces:
\begin{equation}
\label{norm-red-p}
\text{Ind}^{S_{p+q'}}(\rho\otimes\pi_{cusp})=\tau_1\oplus\tau_{-1}
\end{equation}
 (it would be more precise to denote these pieces  by $\tau^{(\rho,\pc)}_1$
and $\tau^{(\rho,\pc)}_{-1}$, but to simplify notation, we   drop the superscripts $(\rho,\pc)$). J.-L. Waldspurger
observed that the normalization of standard intertwining operators which uses C.
M\oe glin  in her definition of $\epsilon_\pi((\rho,a))$  is determined by
the choice of the signs that one attaches to the irreducible pieces in \eqref{norm-red-p}. 

In this paper we work with the classification obtained in \cite{Moe-T}. Therefore, we assume that the normalization of standard intertwining operators in \cite{Moe-Ex} is fixed. This implies that the choice of indexes in \eqref{norm-red-p} is always fixed, when we have situation as above. We shall use this choice of indexes several times in what follows.

 Let us go back to  irreducible square integrable representations  $\pi$ 
of a classical and $\d$ of a general linear group. We shall assume below that 
$$
\text{Ind$^G(\d\o\pi)$}
$$
 reduces. Then $\d$ is selfdual. Therefore, we can find a selfdual irreducible cuspidal representation $\rho$ of a general linear group and $b\in\mathbb Z_{>0}$ such that 
 $$
 \d\cong\d(\rho,b).
 $$
  Now we shall define tempered representation
 $$
 \pi_\d.
 $$
The parabolically induced representations that we  consider in the introduction will be again assumed to be  induced from the appropriate parabolic subgroup, which is  standard with respect to the minimal parabolic subgroup of all upper triangular matrices in the group that we consider. 

  Below, we denote by 
  $$
  \pc
  $$
   the partial cuspidal support of $\pi$ ($\pc$  is the unique irreducible cuspidal representation  of a classical group for which  there exists an irreducible representation $\theta$ of a general linear group such that $\pi\h \text{Ind}(\theta\o\pc)$).

Let  $\tau$ be a representation of a group $G$. Then in the theorem below, $\tau^{\o2}$ will denote the representation $\tau\o\tau$ of the direct product $G\t G$.

\begin{theorem} {\bf  (Definition of $\pi_\d$)} 
Let $\pi$ be an irreducible square integrable representation  of a classical group, 
let $\rho$ be an irreducible selfdual representation  of a general linear group
and let $b$ be a positive integer. Denote
$$
 \d=\d(\rho,b).
 $$
  Assume that 
$$
\text{Ind}^{G}(\delta\otimes \pi)
$$
reduces.
\begin{enumerate}
\item
Suppose
\begin{equation*}
\jrp\cap [1,b]\ne\emptyset.
\end{equation*}
Denote
\begin{equation*}
a
=\max(\jrp \cap[1,b]
).
\end{equation*}
Then there exists a unique irreducible subrepresentation  of Ind$^G(\d\o\pi)$,  denoted by $\pi_\d$,  satisfying
$$
\pi_\d\h \text{Ind}^G(\delta([\nu^{(a-1)/2+1}\rho,\nu^{(b-1)/2}\rho])^{\o2}
\o \lambda)
$$
 for some irreducible representation $\lambda$ of a classical group.  
 
\item
Suppose
\begin{equation*}
\jrp\cap [1,b]=\emptyset.
\end{equation*}

\begin{enumerate}
\item
Let $b$ be even. 
Then there exists a unique irreducible subrepresentation  of Ind$^G(\d\o\pi)$,  denoted by $\pi_\d$,  satisfying
$$
\pi_\d
\h
\text{Ind}^G(\delta([\nu^{1/2}\rho,\nu^{(b-1)/2}\rho])^{\o2}
\o \lambda)
$$
 for some irreducible representation $\lambda$ of a classical group.

\item
Let $b$ be odd.

\begin{enumerate}

\item
Suppose
$$
\jrp\ne\emptyset.
$$
 Denote
$$
a:=\min(\jrp).
$$
Then there exists a unique irreducible subrepresentation  of Ind$^G(\d\o\pi)$,  denoted by $\pi_\d$,  satisfying
$$
\pi_\d
\h
\text{Ind}^G(\delta([\nu\rho,\nu^{(b-1)/2}\rho])^{\o2}\o \delta([\nu\rho,\nu^{(a-1)/2}\rho])
\o \lambda)
$$
 for some irreducible representation $\lambda$  of a classical group.

\item
Suppose
 $$
  \jrp=\emptyset.
  $$
    Then $\rho\r\pc$ reduces. Decompose
\begin{equation}
\label{Eq-not-can}
\rho\r\pc=\tau_1\oplus\tau_{-1}
\end{equation}
into the sum of irreducible (tempered) subrepresentations.
Then there exists a unique irreducible subrepresentation  of Ind$^G(\d\o\pi)$,  denoted by $\pi_\d$,  satisfying
$$
\pi_\d\h \text{Ind}^G(\theta\o\tau_{1}).
$$
for some irreducible representation $\theta$ of a general
linear group.

(Analogously we can define $\pi_{-\d}$ using $\tau_{-1}$ instead $\tau_1$.)

\end{enumerate}
\end{enumerate}
\end{enumerate}
\end{theorem}

One can find Jacquet module definition of representations $\pi_\d$ in section \ref{Primitive}.  Let us note that for the above  parameterization of tempered representations,  we did not need to make any new choice besides the choices that we needed to make for the classification of square integrable representations of classical groups (for square integrable representations we needed to make choices of  $\tau_1$ in \eqref{Eq-not-can}, which  in general  are  not canonical).

\bigskip

Now we shall describe the second  aims of this paper.  
The partially defined function $\ep$ (as well as the partial cuspidal support $\pc$) is defined using embeddings. In general, if we have an irreducible representation $\s$ of a reductive group $G$ and an irreducible representation $\tau$ of a Levi factor $M$ of a parabolic subgroup $P$, 
the fact that $\s$ embeds into $\text{Ind}_P^G(\tau)$, i.e.
$$
\s\h\text{Ind}_P^G(\tau),
$$
implies by Frobenius reciprocity that $\tau$ is a quotient of the corresponding Jacquet module of $\s$ with respect to $P$  (the converse also holds). Then, in particular,  $\tau$ is a subquotient of the corresponding Jacquet module of $\s$. On the other side,  the fact that $\tau$ is a subquotient of the corresponding Jacquet module of $\s$, does not  imply  in general the existence of embedding $\s\h\text{Ind}_P^G(\tau)$ (one can see such examples in Remark \ref{Rm-ex}).

Let us recall that we have fairly good control of subqutients of Jacquet modules of parabolically induced representations (through Geometric Lemma of \cite{B-Z}). The question of exact structure of  Jacquet module is usually much more delicate (see \cite{C-int} already for the case of $SL(2,F)$). Therefore, it would be much more convenient to have characterization of partially defined function  in terms of subqutients of Jacquet modules, instead of quotients.

In this paper we show that in the definition of the partially defined function $\ep$, it is enough to require only the subquotient condition instead of the quotient condition of the corresponding Jacquet module (actually, we shall show more; see section \ref{jac-mod-int}). We shall explain this in more detail  below.

Regarding partial cuspidal support $\pc$ of an irreducible (square integrable) representation $\pi$ of a classical group over $F$, it is easy to show (and it is well-known) that one can define $\pc$ requiring only subquotient (instead of quotient) condition (see Proposition \ref{Pr-pcs}).

 To define partially defined function attached to an irreducible square integrable representation $\pi$ of a classical group over $F$ (defined in \cite{Moe-Ex}; see also \cite{T-inv}), it is enough to consider  three cases of the following theorem (the Jacquet modules that we consider below are all with respect to the standard parabolic subgroups; see section \ref{notation} for more details). 
 
 \begin{theorem}
 {\it
\begin{enumerate}
\item
Suppose that $\jrp$ has at least two elements. Take any $a_-,a\in\jrp$ such that $a_-<a$ and $\{b\in\jrp; a_-<b<a\}=\emptyset$.
   Then
$$
\delta([\nu^{(a_--1)/2+1}\rho, \nu^{(a-1)/2}\rho])\o\sigma
$$
 is a subquotient of the appropriate Jacquet module of $\pi$ for some  irreducible representation $\s$, if and only if $\e_\pi((\rho,a))\e_\pi((\rho,a_-))^{-1}=1$\footnote{By the original definition,
$
\e_\pi((\rho,a))\e_\pi((\rho,a_-))^{-1}=1
$
if and only if there exists a representation $\sigma'$ of a classical group such that
$
\pi\h \text{Ind}(\delta([ \nu^{(a_- -1)/2+1} \rho, 
\nu^{(a-1)/2}\rho])\o \s').
$
}.

\item
Suppose $\jrp\cap 2\mathbb Z\ne \emptyset$. Denote
$$
a_{\pi,\min,\rho}=\min(\jrp).
$$
Then $\ep(\rho,a_{\pi,\min,\rho})$ is defined, and it is 1 if and only if  some  irreducible representation of the form $\delta([\nu^{1/2}\rho, \nu^{(a_{\pi,\min,\rho}-1)/2}\rho])\o\sigma$ is a subquotient of the corresponding Jacquet module of $\pi$\footnote{By the original definition, $\ep(\rho,a_{\pi,\min,\rho})=1$  if and only if $\pi\h\text{Ind}(\delta([\nu^{1/2}\rho, \nu^{(a_{\pi,\min,\rho}-1)/2}\rho])\o\sigma')$ for some irreducible representation $\s'$. 
}.

\item
Suppose $\jrp\cap(1+2\mathbb Z)\ne \emptyset$ and $Jord_\rho(\pc)=\emptyset$ (the last condition is equivalent to the fact that $\rho\r\pc$  reduces). Then  $\rho\r\pc$  reduces into two irreducible nonequivalent representations: $\rho\r \pc=\tau_1\oplus\tau_{-1}.$
 Then for any $k\in \mathbb Z_{>0}$ the representation $\text{Ind}(\delta([\nu\rho,\nu^k\rho])\o\tau_i)$, $i\in\{\pm1\}$, has the unique irreducible subrepresentation, denoted by
$$
\delta([\nu\rho,\nu^k\rho]_{\tau_i};\pc).
$$
This subrepresentation is square integrable.

Denote
$$
a_{\pi,\max,\rho}=\max(\jrp).
$$
Then $\ep(\rho,a_{\pi,\max,\rho})$ is defined and    $\ep(\rho,a_{\pi,\max,\rho})=i$  if and only if an irreducible representation of the form   $\theta\o\delta([\nu\rho, \nu^{(a_{\pi,\max,\rho}-1)/2}\rho]_{\tau_i};\pc)$ is a subquotient of the Jacquet module of $\pi$\footnote{By \cite{T-inv}, $\ep(\rho,a_{\pi,\max,\rho})=i$ if and only if
$
\pi\h\text{Ind}(\theta'\o\delta([\nu\rho, \nu^{(a_{\pi,\max,\rho}-1)/2}\rho]_{\tau_i};\pc))
$
for some representation $\theta'$. 
}.

\end{enumerate}
}
\end{theorem}

Some other useful descriptions in terms of Jacquet modules of partially defined functions,   are also given in section \ref{jac-mod-int}. They are related to the infinitesimal characters.

After  discussion of the role of Jacquet modules in the definition of the invariant $\e_\pi$ (and $\pc$), we shall very briefly discuss the role of Jacquet modules in determining the  invariant $\jp$.
Jacquet modules in this case do not give complete answers, as they do for the other two invariants. For example, we can have an irreducible cuspidal representation $\pi$ of a classical group which has many Jordan blocks, but we certainly can not detect the Jordan blocks from the Jacquet modules of $\pi$ (since all the proper Jacquet modules of a cuspidal representation are trivial, i.e. zero spaces). Nevertheless, from non-trivial Jacquet modules we can get some information about Jordan blocks. An example is Proposition \ref{Pr-jb-jac} (there are possible further results in that direction, but we do not go in that direction in this paper; one can find such results in \cite{J3}).

 As an application of the Jacquet module interpretation of invariants, we give (an expected) description of partially defined function of an irreducible square  integrable representation when one Jordan block of the representation  is deformed. This is one of two important constructions of square integrable representations from \cite{Moe-Ex}. The other  construction is adding two neighbor Jordan blocks (the description of partially defined function for this case is obtained in \cite{Moe-Ex}; see also  \ref{Th-eps=} of this paper).

We explain very briefly the application that we mentioned  above.
Take an irreducible square integrable representation $\pi$ of a classical group and take an irreducible selfdual cuspidal representation $\rho$ of a general linear group .
Let $a\in\jrp$. Suppose that there exists  an integer $b$ of the same parity as $a$, greater than $a$, which satisfies $[a+1,b]\cap\jrp=\emptyset$. Then the representation
\begin{equation}\label{ext}
 \text{Ind}(\delta([ \nu^{(a+1)/2} \rho, 
\nu^{(b-1)/2}\rho])\o \pi)
\end{equation}
has an irreducible square integrable subrepresentation. Denote it by $\pi'$. Then $\pi$ and $\pi'$ have the same partial cuspidal supports, and we know how to get Jordan blocks of one of the representations from the other representation (see Proposition \ref{Pr-jb-jac}). In Theorem \ref{Th-red} we show how to get partially defined function of one of the representations from the  partially defied function of the other representation.

We describe now the content of the paper.
The notation that we use in this paper is introduced in the second section. The third section recalls  some basic results that we use in the paper, in particular about Jordan blocks. We define basic irreducible tempered representations $\pi_\d$ in the fourth section. The fifth section gives a description of all irreducible tempered representations by the basic ones. 
In the sixth section we present some simple observations on the action of the Bernstein center of a factor of  a direct product of two reductive groups, on the representations of the direct product. 
We apply these simple observations in the seventh  section and obtain the Jacquet module interpretation of invariants. The eighth section deals with connection of partially defined functions of   square integrable representations of  smaller and   bigger classical group.
For the convenience of reader, in the appendix  we bring a proof of a result on irreducibility used in this paper, which follows from the paper \cite{Mu3} of G. Mui\'c.

Let us note that at least some of the results of this paper were known to experts. Nevertheless, we present the proofs of the results for which we did not know  written references. Further, description  of tempered representations in some cases has been obtained and used already by some authors (when they where describing irreducible subquotients of parabolically induced representations; see for example  \cite{T-RPI}, \cite{J1}, \cite{BJ} or \cite{Mu4}).

Some of the main topics of this paper 
(like irreducible tempered or square integrable representations of  classical groups over $F$)  show up as main local objects of the recent book \cite{A} of J. Arthur. It would be very important to understand explicitly the relation between the
fundamental classification in \cite{A}
and the problems studied in our paper.
 We do not go in that direction in this paper. We shall say only a few words regarding this. 

Arthur's book contains Langlands classification of irreducible tempered representations of  classical groups over $F$ in characteristic zero (as far as we know, his classification is still conditional, but it is soon expected  to be unconditional). 
Our description of irreducible tempered representations is modulo cuspidal data (i.e. irreducible cuspidal representations and cuspidal reducibilities in the generalized rank one cases). These are   different aspects of the same problem (our description is only one of the steps in understanding the irreducible tempered representations of classical groups). 
For a number of problems of harmonic analysis of classical groups (and automorphic forms), it is important to understand irreducible tempered representations in terms of square integrable ones.
As well, it  is important to understand irreducible square integrable representations in terms of cuspidal representations.
  For example,  the problem of unitarizability is the place where  such information is crucial.
   Such understanding was also important in the case of general linear groups, where the Bernstein-Zelevinsky theory provides  us with such understanding (clearly, the situation there is much simpler).

Arthur's classification of  irreducible square integrable representations is very important for our approach. 
In that classification,    cuspidal representations are very simply recognized  among all the irreducible square integrable  representations. This is done by C. M\oe glin. The  requirement on the admissible homomorphism corresponding to an irreducible cuspidal representation is a very simple condition, "without gaps", while the requirement on the character of the component group is also a simple condition, "being alternate" (for  precise assumptions see   \cite{Moe-mult-one}, and Theorem 2.5.1 there).
From the Langlands parameter  of such a cuspidal representation,  one sees directly all the "exceptional" cuspidal reducibilities, i.e. those ones which are  $\geq 1$ (at least for the odd orthogonal and symplectic groups). This is done by C. M\oe glin and J.-L. Waldspurger (see the remark (ii) in Remarks 4.5.2 of \cite{MW}). The remaining cuspidal reducibilities are controlled already by the local Langlands correspondences for general linear groups. 
Therefore, the Arthur's results  give us precisely the parameters that we are using. 
Let us recall that the case of unitary groups was completed earlier  by C. M\oe glin (in \cite{Moe-Pac}).

Let us mention that C. M\oe glin has obtained in \cite{Moe-mult-one} a parameterization of the packets determined by Arthur's parameters. Since special cases of these packets are tempered,  there is also a parameterization of irreducible tempered representations there. We do not go in this paper into  the relation between the parameterization in her paper and our paper.

After this paper has been completed (and submitted), we have learned for C. Jantzen's paper \cite{J3}. The main aims of  both papers are very close. The  results were independently obtained, and some of them are more or less the same (there is a
slight difference in the choice of parameters of some tempered
representations).
The approaches and the methods of proofs in the papers are very different. Therefore, the
papers are complementary and  give different understanding of the the problems studied in the papers. We are thankful to C. Jantzen for clearing the relation between parameterizations in the papers.

We are very thankful for particularly useful 
discussions on the topic of this paper to C. M\oe glin, and for her explanations. Discussions with M. Hanzer, A. Moy and G. Mui\'c   were also very helpful. The referees gave some very useful suggestions. Parts of this paper were written while the author was the guest of the Hong Kong University of Science and Technology. We are thankful to the University for the hospitality.

\section{Notation}\label{notation}
\setcounter{equation}{0}

In this section we shall  briefly recall  the notation that we
 use  for general linear and classical groups  groups in the paper. This notation we have already used in \cite{Moe-T} and \cite{T-inv} (see also \cite{T-Str}, \cite{Ban}  and \cite{MVW}).
More details regarding this notation can be found in those papers (see \cite{Z} for the case of general linear groups).

We have fixed a local non-archimedean field $F$. 
We  consider in this paper  symplectic, orthogonal and unitary
groups.
  If we consider
unitary groups, then 
$F'$  denotes the separable quadratic extension of $F$ which enters the
definition of the unitary groups. For the other series of groups, we take $F'=F$.
We denote by
$
\theta
$
  the non-trivial $F$-automorphism of
$F'$ if  $F'\ne F$. In the case $F'=F$, $\theta$  denotes the
identity mapping on $F$.
The modulus
character of
$F'$ is  denoted by $|\ \ |_{F'}$, and the character $|\text{det}|_{F'}$ of $GL(n,F')$ is denoted by
$\nu$.

For the group $G$ of rational points of a reductive group  defined
over $F$, the Grothendieck
group of the category Alg$_{\text{\,f.l.}}$($G)$ of  the representations of
$G$ of finite length, is denoted by  $\mathfrak R(G)$ (we  consider only smooth representations in this paper). The Grothendieck group $\mathfrak R(G)$ carries  a natural
ordering $\leq$.
The semi simplification of  
$\tau\in\text{Alg$_{\text{\,f.l.}}$}(G)$ is denoted by  $\text{s.s.}(\tau)$.
For representations  $\pi_1,\pi_2\in\text{Alg$_{\text{\,f.l.}}$}(G)$, the fact
$\text{s.s.}(\pi_1)\leq \text{s.s.}(\pi_2)$ we   write shortly as 
$\pi_1\leq \pi_2$.

For $0\leq k\leq n$, there exists a unique standard (with respect to the minimal parabolic subgroup consisting of the upper triangular matrices in $GL(n,F')$) parabolic
subgroup  $P_{(k,n-k)}=M_{(k,n-k)}N_{(k,n-k)}$ of $GL(n,F')$ whose Levi
factor $M_{(k,n-k)}$ is naturally isomorphic to $GL(k,F')\times
GL(n-k,F')$.
For smooth representations   $\pi_1$ of $GL(n_1,F')$ and
$\pi_2$ of $GL(n_i,F')$, $\pi_1\t
\pi_2$ denotes  the smooth representation of $GL(n_1+n_2,F')$   
     parabolically induced by
$\pi_1\otimes
\pi_2$ from $P_{(k,n-k)}$. Let
$$
R=\oplus_{n\geq 0}\mathfrak R(GL(n,F')).
$$
 We lift $\t$ to a $\mathbb  Z$-bilinear mapping
$R\t R\to R$. We denote this operation again by  $\times$. We can factor the mapping
$\times:R\t R\to R$   through
$R\o R$. The  mapping $R\o R\to R$  which we obtain in this way is denoted by
$m$. 

For a smooth representation   $\pi$ of $GL(n,F')$,
 we denote  by $r_{(k,n-k)}(\pi)$ the
normalized Jacquet module   with respect
to the parabolic subgroup $P_{(k,n-k)}$. 
The comultiplication $m^*:R\to R\o R$ is an additive mapping which  is given by
$
m^*(\pi)=\sum_{k=0}^n \text{s.s.}(r_{(k,n-k)}(\pi))
$
on irreducible representations.
With these two operations, the additive group $R$
becomes a graded Hopf algebra.

Let $\rho$ be an  irreducible cuspidal representation  of a general
linear group (over $F'$), and $n\in \Z_{\geq0}$.
Define
$
[\rho,\nu^n\rho]
:=
\{\rho,\nu\rho,\dots,\nu^n\rho\}.
$ The representation
$\nu^{n}\rho\t\nu^{n-1}\rho\t\dots\t\nu\rho\times\rho$ has a unique
irreducible subrepresentation. This is an essentially square
integrable  representations. It is  denoted 
$\d([\rho,\nu^n\rho])$.
We have
\begin{equation}
\label{Eq-m-seg}
m^*(\d([\rho,\nu^n\rho]))
=
\sum_{i=-1}^n
\d([\rho^{i+1},\nu^n\rho])
\o
\d([\rho,\nu^i\rho])
\end{equation}
(we take formally $\d(\emptyset)$ to be the trivial representation of $GL(0,F)=\{1\}$, which is identity of algebra $R$).
   Denote
$$\delta(\rho,a)=\delta([\nu^{-\frac{a-1}2}\rho,\nu^{\frac{a-1}2}\rho]),
\quad
a\in\mathbb  Z_{\geq0}.
$$

An irreducible essentially square integrable representation
$\delta$ of $GL(n,F')$ can be written as $
\d=\nu^{e(\d)}\d^u
$, where $e(\d)\in \mathbb  R$ and $\d^u$ is an
irreducible unitarizable square integrable representation. This defines $e(\d)$ and $\d^u$.

 In this paper, we shall fix either a series of symplectic, odd or even orthogonal, or unitary groups.

If we fix the symplectic series,  we consider   the Witt
tower of symplectic spaces, and the space in this tower of dimension $2n$  is denoted by  $V_n$ (for symplectic groups we define $Y_0$ to be the trivial vector space $\{0\}$). Then $S_n$ denotes the group of isomorphisms of $V_n$.

In the case of a series of odd orthogonal groups, we fix 
an anisotropic orthogonal vector space $Y_0$ over $F$ of odd
dimension (1 or 3) and  consider the Witt tower based on
$Y_0$. Take $n$ such that $2n+1\geq
$ $\text{dim}\, Y_0$. Then we have exactly one space $V_n$ in the 
tower of dimension  $2n+1$. The special orthogonal group of
this space is denoted by $S_n$. 

For an even-orthogonal series of groups, we fix an anisotropic
orthogonal space $Y_0$ over $F$ of even dimension, and consider the
Witt tower based on $Y_0$. Take $n$ such that $2n \geq \text{dim}_F(Y_0)$. Then
 there is exactly one space  $V_n$ in the  tower of dimension
$2n$. The orthogonal group of $V_n$ is denoted by $S_n$.

In the unitary  case, we consider   unitary groups $U(n,F'/F)$, where $F'$ 
is the separable quadratic extension of $F$ entering the definition of unitary groups. Here we also fix  an
anisotropic unitary space $Y_0$ over $F'$, and  consider the Witt tower of
unitary spaces
$V_n$ based on $Y_0$. Similarly as in the case of orthogonal groups, we have here also odd and even cases.
For $\text{dim}_{F'}(Y_0)$ odd (i.e. 1) and  for  $n$ satisfying  
$2n+1\geq \text{dim}_{F'}(Y_0)$, let   $V_n$ be the space in
the  tower of dimension $2n+1$. Denote the unitary group of this space
 by $S_n$.
For $\text{dim}_{F'}(Y_0)$  even (i.e. 0) and for  $n$ satisfying
$2n\geq \text{dim}_{F'}(Y_0)$, let  $V_n$ be the space in the
 tower of dimension $2n$. Denote its unitary group  by $S_n$.

  We fix a minimal parabolic subgroup in $S_n$ (in this paper we shall consider only standard
parabolic subgroups with respect to this 
minimal parabolic subgroup). One can see in \cite{T-Str}
matrix realizations of the split connected classical groups (and a description of their standard parabolic subgroups).

In what follows,  one  series of groups $\{S_n\}_n$ as above will be fixed. Denote by $n'$
 the Witt index of $V_n$ (i.e. $n'=n-\frac12 \text{dim}_{F'}(Y_0)$ if
$V_n$ is symplectic or even-orthogonal or even-unitary group,
and otherwise $n'=n-\frac12 (\text{dim}_{F'}(Y_0)-1)$). For an integer  $k$ which satisfies 
$0\leq k\leq n'$,  there is a standard parabolic
subgroup $P_{(k)}=M_{(k)}N_{(k)}$ of
$S_n$, whose Levi factor  $M_{(k)}$ is naturally isomorphic to
$GL(k,F')\t S_{n-k}$  
(the group $P_{(k)}$ is the stabilizer of an
isotropic space of dimension $k$). 
Let  $\pi$ and $\s$ be smooth representations of  $GL(k,F')$ and
$S_{n-k}$ respectively. We denote by 
$$
\pi\r\sigma
$$
 the representation parabolically
induced by $\pi\o
\sigma$ (from $P_{(k)}$).
A very useful property of operations $\t$ and $\r$ is
\begin{equation}
\label{iso}
\pi_1\rtimes (\pi_2\rtimes\sigma )\cong (\pi_1\times\pi_2)\rtimes\sigma.
\end{equation}
For a smooth representation $\pi$ of $GL(k,F')$,  denote by $\check\pi$  the representation
$
g\mapsto \tilde\pi(\theta(g))
$
(here $\tilde\pi$ denotes the contragredient representation of
$\pi$). The representation $\pi$  is called  $F'/F$-selfdual if 
$
\pi\cong\check\pi.
$
In the case $F'=F$,  we  say also simply that $\pi$ is selfdual.
 If    
 $\pi$ and $\s$ are representations of finite length, then we have 
\begin{equation}
\label{asso}
\text{s.s.}(\pi\rtimes\sigma)=\text{s.s.}(\check\pi\rtimes\sigma).
\end{equation}
Observe that in the case that $\pi \rtimes \sigma$ is irreducible, we have 
$\pi\rtimes\sigma\cong\check\pi\rtimes\sigma$.
Denote 
$$
R(S)=\underset n \oplus \,\mathfrak R(S_n),
$$
where the above sum runs 
 over all integers $n \geq
\frac12 \,(\text{dim}_{F'}(Y_0)-1)$ if we consider odd-orthogonal or
odd-unitary groups, and otherwise over all $n\geq \frac12\,
\text{dim}_{F'}(Y_0)$). 
Now
$\r$ lifts in a natural way to a bilinear mapping
$
R\t R(S) \to R(S)
$,  denoted again by $\rtimes$. 

 For a smooth representation $\tau$ of $S_n$, the normalized Jacquet module of $\tau$
with respect to $P_{(k)}$ is denoted by
$s_{(k)}(\tau)$.
If $\tau$ is irreducible,
denote
$$
\mu^*(\tau)=\sum_{k=0}^{ n'}
\text{s.s.}\left(s_{(k)}(\tau\right)),
$$ 
where $n'$ denotes
the Witt index of $V_n$. Extend
$\mu^*$ additively to $\mu^*:R(S)\to R\o R(S)$.
Let \
\begin{equation}
\label{Eq-M*}
 M^*= (m \otimes 1) \circ (\, \check{\ }\, \otimes m^\ast)  \circ
\kappa  \circ m^\ast : R\to R\o R,
\end{equation}
 where $\check{\ }\,:R\to R$ is a
group homomorphism determined by the requirement that
$\pi\mapsto\check\pi$ for all irreducible representations $\pi$, and where $\kappa:R\t R
\to R\t R$ maps $\sum x_i\o y_i$ to $\sum y_i\o x_i.$ The action
$\rtimes$ of
$R\otimes R$ on $R\otimes R(S)$ is defined in a natural way.
Then
\begin{equation}
\label{Eq-jm}
\mu^*(\pi \rtimes \sigma)= M^*(\pi) \rtimes \mu^*(\sigma)
\end{equation}
holds  for $\pi\in R$ and $\sigma \in R(S)$.

Let
$\rho$ be an irreducible  $F'/F$-selfdual cuspidal representation of a
general linear group and $x,y\in \mathbb  R$ which satisfy $y-x\in \mathbb 
Z_{\geq0}$. Then \eqref{Eq-m-seg} and \eqref{Eq-jm} imply
\begin{equation}
\label{Eq-M-seg}
M^*\big(\delta([\nu^{x}\rho,\nu^{y}\rho])\big) =
\sum_{i= x-1}^{ y} 
\sum_{j=i}^{ y}
\delta([\nu^{-i}\rho,\nu^{-x}\rho]) 
 \times
\delta([\nu^{j+1} \rho,\nu^{y}\rho]) \otimes
\delta([\nu^{i+1} \rho,\nu^{j}\rho]),
\end{equation}
where $y-i,y-j\in \mathbb  Z_{\geq 0}$ in the above sums. The part of $M^*\big(\delta([\nu^{x}\rho,\nu^{y}\rho])\big)$ contained in $R\o R_0$ will be denoted by $M_{GL}^*\big(\delta([\nu^{x}\rho,\nu^{y}\rho])\big)\o1$. Then
\begin{equation}
\label{Eq-MGL-seg}
M_{GL}^*\big(\delta([\nu^{x}\rho,\nu^{y}\rho])\big) =
\sum_{i= x-1}^{ y} 
\delta([\nu^{-i}\rho,\nu^{-x}\rho]) 
 \times
\delta([\nu^{i+1} \rho,\nu^{y}\rho])
\end{equation}
(again  $y-i\in \mathbb  Z_{\geq 0}$ in the above sum).

Now we shall recall  the definition of Jordan blocks attached to an irreducible square integrable representation of a classical group. First we shall define representations $R_k$ of the $L$-group of  $F$-group $GL(k,F')$.
Suppose that we  consider the symplectic series of 
groups (resp. one of orthogonal series of groups), Let $k\in\mathbb  Z_{>0}$. 
Then by $R_k$ will be denote the representation  of
$GL(k,\mathbb  C)$ on $\wedge^2\mathbb  C^k$ (resp., $\text{Sym}^2\mathbb  C^k$).

Now suppose that we  consider a series of unitary groups. Then the  $L$-group
of $F$-group
 $GL(k,F')$ is isomorphic to a semidirect
product
$
\big(GL(k,\mathbb  C)\times GL(k,\mathbb  C)\big) \rightthreetimes
Gal(F'/F).
$
The non-trivial element  $\theta$ of $Gal(F'/F)$ acts on the normal
subgroup $GL(k,\mathbb  C)\times GL(k,\mathbb  C)$ by
$
\theta(g_1,g_2,1)\theta^{-1}=(^t\!g_2^{-1},^t\!\!g_1^{-1},1)
$, where $^t\!g$ denotes the transposed matrix of $g$ (see \cite{Moe-Ex}).
For $\eta\in\{\pm1\}$ denote by  $R_{k}^{(\eta)}$   the
representation of the above $L$-group (of $GL(k,F')$) on
$\text{End}_{\mathbb  C}(\mathbb  C^{k})$, determined  by
$
(g_1,g_2,1)\,u=g_1\, u\, ^t\!g_2
$
 and 
$
(1,1,\theta)\,u=\eta\, ^t\!u.
$
Suppose that $S_n$ is a series of (unitary) groups such that
dimensions of $V_n$ are even (resp. odd). Then by  
$R_{k}$ will be denoted the representation $R_{k}^{(1)}$ (resp.
$R_{k}^{(-1)}$).

Now we have the following  definition  of the Jordan blocks (from \cite{Moe-Ex};   $L(\rho,R_{d_\rho},s)$ denotes  the $L$-function
defined by F. Shahidi).

 \begin{definition}
 \label{Def-jb} The set
$Jord(\pi)$ of Jordan blocks attached to an irreducible square integrable representation $\pi$ of a classical group $S_n$ is the set of all pairs
$(\rho,a)$ where
$\rho$ is an irreducible
$F'/F$-selfdual cuspidal representation of 
$GL(d_\rho,F')$  and
$a\in
\mathbb  Z_{>0}$, such that:
\begin{enumerate}
\item[(J1)]
$ a \text{ is even if } L(\rho,R_{d_\rho},s) \text{ has a pole at
}s=0, \text{ and odd otherwise,}
$

\item[(J2)]
$
\d(\rho,a)\r\pi \text{ is irreducible.}
$
\end{enumerate}
\enddefinition

Denote
$
Jord_\rho(\pi)=\{a;(\rho,a)\in Jord(\pi)\}
$
($\rho$ is  an
irreducible $F'/F$-selfdual cuspidal representation of a general linear group).

\end{definition}

In this paper we shall assume   that the following (basic)
assumption holds for any 
 irreducible cuspidal representation $\pi_c$ of any $S_q$, and for any
 irreducible $F'/F$-selfdual cuspidal representation $\rho$ of any
$GL(p,F')$:
\begin{enumerate}
\item[(BA)] If we denote by
$$
a_{\pi_c,\max,\rho}= 
\begin{cases}
\ \text{max}\, Jord_\rho(\pi_{c}) 
&
\text{ if } 
Jord_\rho(\pi_{c}) \ne\emptyset,
\\
\ 0 
&
\text{ if } 
L(\rho,R_{d_\rho},s) \text{ has a pole at $s=0$} \text{ and }
Jord_\rho(\pi_{c}) =\emptyset, 
\\ -1 
&
 \text{ otherwise. }
\end{cases}
$$
\end{enumerate}
then 
$$
\nu^{\pm(1+a_{\pi_c,\max,\rho})/2}\rho \rtimes\pi_{c} 
$$
$\text{ reduces}$.

\begin{remark}\label{rem-BA}
  The basic assumption is known to hold in some cases, while for (the most) of other cases, it is expected to be known fact   very soon (when  the facts on which relays \cite{A} become available).

  For unitary groups, it follows from C. M\oe glin's paper \cite{Moe-Pac}. Namely, by \cite{Si} we know that there is only one non-negative exponent $x$ for which $\nu^x\rho\r\pi_c$ is reducible. Propositions 3.1 of \cite{Moe-Pac} gives that this exponent is integer or half-integer. If it is 0 or 1/2, then  Proposition 13.2 of \cite{T-RPI} and \cite{Moe-fin} imply that $Jord_\rho(\pi_{c})=\emptyset$. Suppose that  $\nu^x\rho\r\pi_c$ reduces for some $x\geq 1$. Now Theorem 13.2 of \cite{T-RPI} implies that $\d(\rho,n)\r\pi_c$ is irreducible for all positive integers of the parity same as the parity of $2x$, while for the other parity we have irreducibility for $n\leq 2x-1$, and reducibility otherwise. Note that now the unique irreducible subrepresentation of $\nu^x\rho\r\pi_c$ is square integrable. Denote it by $\pi$. Then the extended cuspidal support of $\pi$ contains $\nu^x\rho$. Now Proposition
   5.6 of \cite{Moe-Pac} implies that $2x+1$ has even parity if the $L$-function from the Definition  \ref{Def-jb} has a pole at $s=0$, and odd otherwise. Therefore, the basic assumption holds for unitary groups.
   
   Observe that \cite{Moe-Pac} gives much more then the basic assumption. It gives the full classification of irreducible square integrable representations, and it gives also the full classification  of the (subfamily) of cuspidal representations of these groups (moreover, from the parameters on the Galois side one can see what are the cuspidal reducibilities).

    For the odd orthogonal and symplectic groups,  it follows from  remark (ii) in Remarks 4.5.2 of the paper \cite{MW} of C. M\oe glin and J.-L. Waldspurger and the  book \cite{A} of J. Arthur (Theorem 1.5.1). We are very thankful to C. M\oe glin for explaining this to us.

\end{remark}

In this section we have recalled  the definition of invariants $\jp$ and $\pc$. The definition of the third invariant $\ep$, convenient for our purposes, can be found in \cite{T-inv} (see also section \ref{jac-mod-int} of this paper). Let us  recall that $\ep$ is defined on a subset of $\jp\ \cup\ \jp\t\jp$, and that it takes values in $\{\pm1\}$.

\section{Two basic preliminary facts about Jordan blocks and partially defined functions}\label{prelim}
\setcounter{equation}{0}

Here we shall recall  two basic facts about Jordan blocks and partially defined functions which we shall use often in the paper.

First we shall recall  a general fact regarding Jordan blocks (Proposition  2.1 of \cite{Moe-T}):

\begin{proposition}\label{Pr-jb-main} Let $\pi'$ be an irreducible square
integrable representation of
$S_q$ and let
$x,y \in (1/2){ {\mathbb Z}}$ such that $x-y\in {\mathbb Z}_{\geq 0}$. Let
$\rho$ be an $F'/F$-selfdual cuspidal 
(unitarizable)
representation of
$GL(d_\rho)$. We assume that $x$,
$y\in\Z$ if and only if $L(\rho,R_{d_\rho},s)$ has no pole
at
$s=0$. Further, suppose that there is an irreducible square
integrable representation
$\pi$ embedded in the induced representation
\begin{equation}\label{Eq-jb1-}
\pi \hookrightarrow 
\d([\nu^{y}\rho,  \nu^x \rho])\rtimes \pi'.
\end{equation}
or more generally, into
\begin{equation}\label{Eq-jb1}
\pi \hookrightarrow \nu^x \rho\times \cdots \times
\nu^{x-i+1}\rho
\times \cdots \times \nu^y \rho\rtimes \pi'.
\end{equation}
Then holds:

\begin{enumerate}
\item
If $y>0$, then 
$$ 
Jord(\pi)=
\left(Jord(\pi')\ \backslash\ \{(\rho,2y-1)\}\right)\cup
\{(\rho,2x+1)\}.
$$ 
Further, $2y-1\in Jord_\rho(\pi')$ if $y\geq 1$.

\item If $y\leq 0$ , then
$$ 
Jord(\pi)=Jord(\pi')\cup \{(\rho,2x+1),(\rho,-2y+1)\}.
$$ 
In particular, $2x+1$ and $-2y+1$ are not in
$Jord_\rho(\pi')$.
\end{enumerate}
\end{proposition}

The following theorem follows directly from Lemmas 5.4, 5.5 and (ii) of Proposition 6.1 (all) from \cite{Moe-Ex}:

\begin{theorem}\label{Th-eps=} Let $\pi'$ be an irreducible square
integrable representation of
$S_q$. Take
$a_,a_- \in  \mathbb Z$   such that $a-a_-\in {2\mathbb Z}_{>0}$.  Let
$\rho$ be an $F'/F$-selfdual cuspidal  representation of
$GL(d_\rho)$. Suppose that $a_-$,
$a\in1+2\Z$ if and only if $L(\rho,R_{d_\rho},s)$ has no pole
at
$s=0$. Assume that 
$$
\text{\rm Jord}_\rho(\pi')\cap [a_-,a]=\emptyset.
$$
 Let
$\pi$ be an  irreducible subrepresentation of
\begin{equation}\label{Eq-eps=}
\delta([ \nu^{-(a_--1)/2} \rho, 
\nu^{(a-1)/2}\rho])
\rtimes \pi'.
\end{equation}
Then $\pi$ is square integrable and
$$ 
\jrp=Jord_\rho(\pi')\cup \{(\rho,a_-),(\rho,a)\}.
$$ 
Further, for $(\rho',b)\in Jord_\rho(\pi')$, $\e_{\pi}(\rho',b)$ is defined if and only if $\e_{\pi'}(\rho',b)$ is defined. If they are defined, then
$$
\e_{\pi}(\rho',b)=\e_{\pi'}(\rho',b).
$$
If there exists $(\rho',c)\in \text{\rm Jord}(\pi')$ with $c \ne b$, then
$$
\e_{\pi}(\rho',b)\e_{\pi}(\rho',c)^{-1}=\e_{\pi'}(\rho',b)\e_{\pi'}(\rho',c)^{-1}.
$$
\end{theorem}

\begin{definition}
\label{reduction=}
 The square integrable representation $\pi'$ in the above theorem is uniquely determined (up to an equivalence) by $\pi$, and we shall denote it by
$$
\pi'=\pi^{-\{(\rho,a_-),(\rho,a)\}}.
$$
\end{definition}
 
 Now we shall recall  Lemma 5.3 of \cite{Moe-T}:

\begin{lemma}\label{Le-sr} Let $\pi$ be an irreducible square integrable
representation and suppose that there exists
$a\in Jord_\rho(\pi)$ such that $a+2\not\in Jord_\rho(\pi)$. Then
$$
\nu^{(a+1)/2}\rho\r \pi
$$ reduces. Further, it contains a unique irreducible subrepresentation
\end{lemma}

We  recall next the elementary Lemma 3.2  of \cite{Moe-T} which shall be used often without mentioning.

\begin{lemma}\label{Le-e2}  Let $\pi$ be an irreducible representation of a
reductive $p$-adic group and let $P=MN$ be a parabolic subgroup of
$G$. Suppose that $M$ is a direct product of two reductive subgroups
$M_1$ and $M_2$. Let
$\tau_1$ be an irreducible representation of $M_1$, and let $\tau_2$
be a representation of
$M_2$. Suppose
$$
\pi\h \text{Ind}_P^G(\tau_1\o \tau_2).
$$ Then there exists an irreducible representation $\tau_2'$ such that
$$
\pi\h \text{Ind}_P^G(\tau_1\o \tau_2').
$$
\end{lemma}

We shall often use the following 

\begin{proposition}\label{Pr-jb-jac} Let $\pi$ be an irreducible square
integrable representation of  $S_q$. 
\begin{enumerate}
\item
Suppose that $\nu^x\rho\otimes\tau$ is an
irreducible subquotient of a standard Jacquet module of $\pi$,  where
$\rho$ is an irreducible  $F'/F$-selfdual cuspidal representation of
$GL(p,F')$,
$x\in \mathbb R$, and $\tau$ is an irreducible representation of
$S_{q'}$. Then
$$ 
 (\rho,2x+1)\in Jord(\pi).
$$
 
 \item More generally, let $\sigma\otimes\tau$ be an
irreducible subquotient of a standard Jacquet module of $\pi$,  where
$\s$ is an irreducible   representation of
$GL(p,F')$,
 and $\tau$ is an irreducible representation of
$S_{q'}$. Then there exists an irreducible cuspidal representation $\rho'$ in the cuspidal support of $\s$ such that if we write $\rho'=\nu^x\rho$ with $x\in\mathbb R$ and $\rho$ unitarizable, then
$$ 
 (\rho,2x+1)\in Jord(\pi).
$$
 
 \end{enumerate}
\end{proposition}

One can find the definition of the cuspidal support in Remark \ref{GL-inf} in what follows.
 
 \begin{proof}
 The claim (1) is just Lemma 3.6 of \cite{Moe-T} (for a slightly different argument observe that (3) of Corollary \ref{Cor-bc} implies that $\pi\h \nu^x\rho\r\s$ for some irreducible representation $\s$; now Remark 5.1.2 of \cite{Moe-Ex} implies (1)).

 For the second claim, observe that the
  the transitivity of Jacquet modules implies that there exists an irreducible cuspidal representation $\rho'$ in the cuspidal support of $\s$ and an irreducible representation $\mu$ of a classical group such that $\rho'\o\mu$ is a subquotient of the Jacquet module of $\pi$. Now (1) implies (2).
 \end{proof}

Let $(\rho,a)$ be a Jordan block of an irreducible square integrable representation $\pi$ of a classical group. If there exists the element  $b$ in $\jrp$ such that $b<a$, and
$
\{x\in\jrp;b<x<a\}=\emptyset,
$
then 
we shall denote $b$ by 
$$
a_-.
$$
In this case we shall say that $a\in\jrp$ has $a_-$.

\section{First tempered reductions}\label{Primitive}
\setcounter{equation}{0}

We recall  the definition of cuspidal support in the case of general linear groups.  If an irreducible representation $\pi$ of a general linear group is a subquotient of $\rho_1\t\dots\t\rho_k$, then the multiset $(\rho_1,\dots,\rho_k)$ is called the cuspidal support of $\pi$, and denoted by
$$
\text{supp}(\pi).
$$

\begin{lemma}
\label{Le-def-main}
 Let $\pi$ be an irreducible square integrable representation
of a classical group, $ \rho$ an $F'/F$-selfdual irreducible cuspidal representation of a general linear group and $b\in\mathbb Z_{>0}$ such that 
$$
\d(\rho,b)\r\pi
$$
reduces (equivalently, $\rho$ and $b$ satisfy (J1),  and $b\not\in\jrp$).

Suppose
\begin{equation}
\label{Eq-PR1}
\jrp\cap [1,b]\ne\emptyset
\end{equation}
and denote
\begin{equation}
\label{Eq-PR2}
a
=\max(\jrp \cap[1,b]
).
\end{equation}
Then:

\begin{enumerate}
\item
There exists an irreducible representation  $\theta\o\sigma$   satisfying
$$
\theta\o\sigma\leq \mu^*(\d(\rho,b)\r\pi)
$$
   and
\begin{equation}
\label{supp}
 \text{supp}(\theta)= [\nu^{(a-1)/2+1}\rho,\nu^{(b-1)/2}\rho]
+[\nu^{(a-1)/2+1}\rho,\nu^{(b-1)/2}\rho].
\end{equation}
Further, such $\theta$ and $\s$ are unique, and satisfy
$$
\theta\o\sigma \cong
\d([\nu^{(a-1)/2+1}\rho,\nu^{(b-1)/2}\rho])^2\o\d(\rho,a)\r\pi.
$$

\item
The multiplicity of
$$
\d([\nu^{(a-1)/2+1}\rho,\nu^{(b-1)/2}\rho])^2\o\d(\rho,a)\r\pi
$$
in
$$
\mu^*(\d(\rho,b)\r\pi)
$$
is one.

\item
$
\d([\nu^{(a-1)/2+1}\rho,\nu^{(b-1)/2}\rho])^2\o\d(\rho,a)\r\pi
$ is a direct summand in the corresponding Jacquet module of
$\d(\rho,b)\r\pi$.

\end{enumerate}
\end{lemma}

\begin{proof} Observe that $b\geq 3$.
 Suppose that $\theta\o\s$ is some irreducible representation satisfying the assumption \eqref{supp}.
 We shall now consider the formula for $\mu^*(\d(\rho,b)\r\pi)$, determine the terms of the sums where representation of type $\theta\o\s$ can  be a subquotient, determine $\theta\o\s$ and determine the multiplicity of $\theta\o\s$.
 
  From \eqref{Eq-jm} and \eqref{Eq-M-seg} follows  that $\mu^*(\d(\rho,b)\r\pi)$ is
   \begin{equation}
\label{Eq-main}
\left(\sum_{i= -(b-1)/2-1}^{ (b-1)/2} 
\sum_{j=i}^{ (b-1)/2}
\delta([\nu^{-i}\rho,\nu^{(b-1)/2}\rho]) 
 \times
\delta([\nu^{j+1} \rho,\nu^{(b-1)/2}\rho]) \otimes
\delta([\nu^{i+1} \rho,\nu^{j}\rho])\right)\r\mu^*(\pi).
\end{equation}
We shall now analyze in the above sums  indexes $i$ and $j$   for which we have possibility to get a representation $\theta\o\s$ satisfying \eqref{supp} for a subquotient.
The condition on the support of $\theta$ implies that if $\theta\o\s$ as above is a subquotient, then the indexes must satisfy 
\begin{equation}
\label{Eq-leq}
(a-1)/2+1\leq -i \text{ and } (a-1)/2+1\leq j+1.
\end{equation}

Suppose now that $\theta \o\s$ as above is a subquotient of some term corresponding to indexes $i$ and $j$ which satisfy
\begin{equation}
\label{Eq-<}
(a-1)/2+1< -i \text{ or }(a-1)/2+1< j+1.
\end{equation}
Then there exists an irreducible subquotient $\gamma\o\lambda$ of $\mu^*(\pi)$ such that 
$$
\theta\o\s\leq \delta([\nu^{-i}\rho,\nu^{(b-1)/2}\rho]) 
 \times
\delta([\nu^{j+1} \rho,\nu^{(b-1)/2}\rho])\t\gamma\o \delta([\nu^{i+1} \rho,\nu^{j}\rho])\r \lambda.
$$
 The assumption on the cuspidal support of $\theta$ implies supp$(\gamma)=[\nu^{(a-1)/2+1}\rho,\nu^{-i-1}\rho]+[\nu^{(a-1)/2+1}\rho,\nu^{j}\rho]$. Recall that our assumption \eqref{Eq-<} implies supp$(\gamma)\ne\emptyset$. Now  (2) of Proposition \ref{Pr-jb-jac} implies the existence of 
 $$
 \nu^x\rho\in 
 [\nu^{(a-1)/2+1}\rho,\nu^{-i-1}\rho]\cup[\nu^{(a-1)/2+1}\rho,\nu^{j}\rho]
 $$
 such that $2x+1\in\jrp$. Observe that $(a-1)/2+1\leq x$ implies 
 $$
 a+2\leq 2x+1.
 $$
 From the other side, $i\geq -(b-1)/2-1$ implies $x\leq -i-1\leq (b-1)/2$ which further implies $2x+1\leq b$, and $j\leq (b-1)/2$ implies again $2x+1\leq 2j+1\leq b$. 
 
 Therefore, we have obtained that \eqref{Eq-<} implies $[a+2,b]\cap \jrp\ne \emptyset$. This contradicts the assumption \eqref{Eq-PR2}. This contradiction implies that $\theta \o\s$ as above can not be a subquotient of some term corresponding to indexes $i$ and $j$ which satisfy \eqref{Eq-<}.
 
 Thus, if $\theta \o\s$ as above is a subquotient of some term corresponding to indexes $i$ and $j$, then these indexes must satisfy
  $
(a-1)/2+1= -i \text{ and } (a-1)/2+1= j+1$. 

Observe that if we take in \eqref{Eq-main} the term from the double sum corresponding to indexes $
(a-1)/2+1= -i \text{ and } (a-1)/2+1= j+1$, and from $\mu^*(\pi)$ the term $1\o\pi$, then we get in \eqref{Eq-main} the term
$$
\delta([\nu^{(a-1)/2+1}\rho,\nu^{(b-1)/2}\rho])^2
\o \delta([\nu^{-(a-1)/2} \rho,\nu^{(a-1)/2}\rho])\r \pi.
$$
Obsetve that this representation is irreducible, and satisfies \eqref{supp}. This shows that that there is at least on $\theta\o\s$ as above.

Suppose now  that $\theta \o\s$ is any representation as above (i.e. satisfying  \eqref{supp}), which is a subquotient of $\mu^*(\d(\rho,b)\r\pi)$ (we have seen that there exists  at least one such representation).
Then 
$$
\theta\o\s\leq \delta([\nu^{(a-1)/2+1}\rho,\nu^{(b-1)/2}\rho])^2 \t\gamma\o \delta([\nu^{-(a-1)/2} \rho,\nu^{(a-1)/2}\rho])\r \lambda
$$
for some $\gamma\o\lambda\leq \mu^*(\pi)$.
Now the assumption on the cuspidal support implies that $\gamma=1$. Therefore 
$1\o\lambda \leq \mu^*(\pi)$, which implies $1\o\lambda\cong1\o\pi$. Thus
$$
\theta\o\s\leq 
\delta([\nu^{(a-1)/2+1}\rho,\nu^{(b-1)/2}\rho])^2
\o \delta([\nu^{-(a-1)/2} \rho,\nu^{(a-1)/2}\rho])\r \pi.
$$
Since the  representation on the right hand side is irreducible,  we have the equality above. Further, since the multiplicity of $1\o\pi$ in $\mu^*(\pi)$ is one, we conclude that the multiplicity of $\theta \o\s$ in $\mu^*(\d(\rho,b)\r\pi)$ is one.

This completes the proof of (1) and (2) of the lemma. Further, using the properties of the Bernstein center (see section \ref{elem-b-c}) (3) follows from (1).
\end{proof}

\begin{corollary}
\label{Cor-1}
Let the assumptions be the same as in the previous lemma. Then the following  requirements on an irreducible subquotient $\tau$ of $\d(\rho,b)\r\pi$ are equivalent:

\begin{enumerate}

\item $\tau$ embeds into $\delta([\nu^{(a-1)/2+1}\rho,\nu^{(b-1)/2}\rho])^2
\t \delta([\nu^{-(a-1)/2} \rho,\nu^{(a-1)/2}\rho])\r \pi$.

\item $\tau$ embeds into $\delta([\nu^{(a-1)/2+1}\rho,\nu^{(b-1)/2}\rho])^2
\r \lambda$ for some irreducible representation $\lambda$.

\item $\tau$ embeds into $\theta
\r \lambda$ for some irreducible representations $\theta$ and $\lambda$, such that supp$(\theta)=2[\nu^{(a-1)/2+1}\rho,\nu^{(b-1)/2}\rho].$

\item $\tau$ has $\delta([\nu^{(a-1)/2+1}\rho,\nu^{(b-1)/2}\rho])^2
\o\delta([\nu^{-(a-1)/2} \rho,\nu^{(a-1)/2}\rho])\r \pi$ for  a subquotient of the corresponding  Jacquet module.

\item $\tau$ has $\delta([\nu^{(a-1)/2+1}\rho,\nu^{(b-1)/2}\rho])^2
\o \lambda$ for  a subquotient of the corresponding  Jacquet module,  for some irreducible representation $\lambda$.

\item $\tau$ has $\theta
\o \lambda$ for  a subquotient of the corresponding  Jacquet module, for some irreducible representations $\theta$ and $\lambda$, such that supp$(\theta)=2[\nu^{(a-1)/2+1}\rho,\nu^{(b-1)/2}\rho].$ \qed

\end{enumerate} 

\end{corollary}

\begin{definition}
\label{Def-1}
With the assumptions  as in the previous lemma,  the  irreducible subquotient $\tau$ of $\d(\rho,b)\r\pi$ satisfying the equivalent requirements of the above corollary, will be denoted by
$$
\pi_{\d(\rho,b)}.
$$
The other irreducible subrepresentation of $\d(\rho,b)\r\pi$ will be denoted by
$$
\pi_{-\d(\rho,b)}.
$$

\end{definition}

\begin{lemma}
\label{Le-def-even}
 Let $\pi$ be an irreducible square integrable representation
of a classical group, $ \rho$ an $F'/F$-selfdual irreducible cuspidal representation of a general linear group and 
$$
b\in2\mathbb Z_{>0}
$$
 such that 
$$
\d(\rho,b)\r\pi
$$
reduces (i.e., $\rho$ and $b$ satisfy (J1),  and $b\not\in\jrp$).

Suppose
\begin{equation}
\label{Eq-PR3}
\jrp\cap [1,b]=\emptyset.
\end{equation}
 Then

\begin{enumerate}
\item
There exists an irreducible representation  $\theta\o\s$  satisfying $\theta\o\sigma\leq \mu^*(\d(\rho,b)\r\pi)$ and
$$
 \text{supp}(\theta)= [\nu^{1/2}\rho,\nu^{(b-1)/2}\rho]
+[\nu^{1/2}\rho,\nu^{(b-1)/2}\rho].
$$
Such $\theta$ and $\s$ are unique, and holds
$$
\theta\o\sigma \cong
\d([\nu^{1/2}\rho,\nu^{(b-1)/2}\rho])^2\o\pi.
$$

\item
 The multiplicity of
$$
\d([\nu^{1/2}\rho,\nu^{(b-1)/2}\rho])^2\o\pi
$$
in
$$
\mu^*(\d(\rho,b)\r\pi)
$$
is one.

\item
$
\d([\nu^{1/2}\rho,\nu^{(b-1)/2}\rho])^2\o\pi
$
is a direct summand in the corresponding Jacquet module of
$
\d(\rho,b)\r\pi.
$
\end{enumerate}
\end{lemma}

\begin{proof}
The proof proceeds in a very similar way as the proof of Lemma \ref{Le-def-main}. We shall very briefly sketch the proof. To get any $\theta\o\s$ satisfying (1) from \eqref{Eq-main}, we must take $1/2\leq -i$ and $1/2\leq j+1$. If at least one of these two inequalities is strict, we would get that some positive even integer $k\leq b$ is in $\jrp$, which contradicts the assumption of the lemma. Now the proof proceeds in exactly the same way as the proof of Lemma \ref{Le-def-main}. 
\end{proof}

\begin{corollary}
\label{Cor-2}
Let the assumptions be the same as in the previous lemma. Then the following  requirements on an irreducible subquotient $\tau$ of $\d(\rho,b)\r\pi$ are equivalent:

\begin{enumerate}

\item $\tau$ embeds into $\delta([\nu^{1/2}\rho,\nu^{(b-1)/2}\rho])^2
\r \pi$.

\item $\tau$ embeds into $\delta([\nu^{1/2}\rho,\nu^{(b-1)/2}\rho])^2
\r \lambda$ for some irreducible representation $\lambda$.

\item $\tau$ embeds into $\theta
\r \lambda$ for some irreducible representations $\theta$ and $\lambda$, such that supp$(\theta)=2[\nu^{1/2}\rho,\nu^{(b-1)/2}\rho].$

\item $\tau$ has $\delta([\nu^{1/2}\rho,\nu^{(b-1)/2}\rho])^2
\o \pi$ for  a subquotient of its Jacquet module.

\item $\tau$ has $\delta([\nu^{1/2}\rho,\nu^{(b-1)/2}\rho])^2
\o \lambda$ for  a subquotient of its Jacquet module,  for some irreducible representation $\lambda$.

\item $\tau$ has $\theta
\o \lambda$ for  a subquotient of its Jacquet module, for some irreducible representations $\theta$ and $\lambda$, such that supp$(\theta)=2[\nu^{1/2}\rho,\nu^{(b-1)/2}\rho].$ \qed

\end{enumerate} 

\end{corollary}

\begin{definition}

With the assumptions  as in the previous lemma,  the  irreducible subquotient $\tau$ of $\d(\rho,b)\r\pi$ satisfying the equivalent requirements of the above corollary, will be denoted by
$$
\pi_{\d(\rho,b)}.
$$
The other irreducible subrepresentation of $\d(\rho,b)\r\pi$ will be denoted by
$
\pi_{-\d(\rho,b)}.
$

\end{definition}

\begin{lemma}\label{Le-def-odd-2}
 Let $\pi$ be an irreducible square integrable representation
of a classical group, $ \rho$ an $F'/F$-selfdual irreducible cuspidal representation of a general linear group and 
$$
b\in1+2\mathbb Z_{\geq 0}
$$
 such that 
$$
\d(\rho,b)\r\pi
$$
reduces (i.e., $\rho$ and $b$ satisfy (J1),  and $b\not\in\jrp$).

Suppose
\begin{equation}
\label{Eq-PR4}
\jrp\cap [1,b]=\emptyset
\end{equation}
and
$$
\jrp\ne \emptyset.
$$
 Denote
$$
a:=\min(\jrp).
$$
Then:

\begin{enumerate}
\item
There exists  an irreducible representation  $\theta\o\s$ satisfying $\theta\o\sigma\leq \mu^*(\d(\rho,b)\r\pi)$  and
$$
 \text{supp}(\theta)= [\nu\rho,\nu^{(a-1)/2}\rho] +
[\nu\rho,\nu^{(b-1)/2}\rho]+[\nu\rho,\nu^{(b-1)/2}\rho].
$$
Such $\theta$ and $\s$ are unique. More precisely, we have
$$
\theta\cong \d( [\nu\rho,\nu^{(a-1)/2}\rho])\t
\d([\nu\rho,\nu^{(b-1)/2}\rho])^2.
$$
Further, (ii) of Lemma 5.4.2 in \cite{Moe-Ex}  implies that
$\pi\h\d([\nu\rho,\nu^{(a-1)/2}\rho])\r\pi'$ for some irreducible square integrable
representation $\pi'$ of a classical group. Then $\rho\r\pi'$ is irreducible and
$$
\sigma\cong\rho\r\pi'.
$$

\item
The multiplicity of $\d( [\nu\rho,\nu^{(a-1)/2}\rho])\t
\d([\nu\rho,\nu^{(b-1)/2}\rho])^2\o\rho\r\pi'$   in
$\mu^*(\d(\rho,b)\r\pi)$ is one. 

\item
 $\d( [\nu\rho,\nu^{(a-1)/2}\rho])\t
\d([\nu\rho,\nu^{(b-1)/2}\rho])^2\o\rho\r\pi'$ is a direct summand in the corresponding
Jacquet module of $\d(\rho,b)\r\pi$.

\end{enumerate}

\end{lemma}

\begin{remark}
The representation $\pi'$ in (2) of the above lemma has explicit description by the admissible triple (see section \ref{behaviour}).
\end{remark}

\begin{proof} (1) Clearly, $a\geq 3$. Observe that from
$\pi\h\d([\nu\rho,\nu^{(a-1)/2}\rho])\r\pi'$ and (1) of Proposition \ref{Pr-jb-main} 
follow
$$
Jord(\pi)=\big(Jord(\pi')\backslash\{(\rho,1)\}\big) \cup\{(\rho,a)\},
$$
i.e.
\begin{equation}
\label{Eq-jbpi'}
Jord(\pi')=\big(Jord(\pi)\backslash\{(\rho,a)\}\big) \cup\{(\rho,1)\}.
\end{equation}

We have obviously
$$
\mu^*(\d(\rho,b)\r\pi)
\leq 
\mu^*(\d(\rho,b)\t
\d([\nu\rho,\nu^{(a-1)/2}\rho])\r\pi')
=
\hskip30mm
$$
$$
\hskip50mm
M^*(\d(\rho,b))\t
M^*(\d([\nu\rho,\nu^{(a-1)/2}\rho]))\r\mu^*(\pi').
$$
Now we shall analyze when we can get  an irreducible term  $\theta\o\sigma$ such that
$$
 \text{supp}(\theta)= [\nu\rho,\nu^{(a-1)/2}\rho] +
[\nu\rho,\nu^{(b-1)/2}\rho]+[\nu\rho,\nu^{(b-1)/2}\rho]
$$
in the right hand side of the above inequality.
Considering the cuspidal support of $\theta$ and Jordan blocks of $\pi'$ (see \eqref{Eq-jbpi'}), we see that to get $\theta\o\s$ as a subquotient, we
 must take from 
$M^*(\d(\rho,b))$ the term
$\d([\nu\rho,\nu^{(b-1)/2}\rho])^2\o\rho$, from $M^*(\d([\nu\rho,\nu^{(a-1)/2}\rho])$ the term
$\d([\nu\rho,\nu^{(a-1)/2}\rho])\o
1$  and from  $\mu^*(\pi')$ the term
$1\o\pi'$. Thus
\begin{equation}
\label{Eq-*}
\theta\o\s\cong \d( [\nu\rho,\nu^{(a-1)/2}\rho])\t
\d([\nu\rho,\nu^{(b-1)/2}\rho])^2\o\rho\r\pi'.
\end{equation}
This proves multiplicity one in $\mu^*(\d(\rho,b)\t
\d([\nu\rho,\nu^{(a-1)/2}\rho])\r\pi')$. Note that the irreducibility of $\rho\r\pi'$ follows from \eqref{Eq-jbpi'}. Uniqueness of $\s$ will follow from section 8.

From the other side 
$
\pi\h\d([\nu\rho,\nu^{(a-1)/2}\rho])\r\pi'
$
  implies that $\d([\nu\rho,\nu^{(a-1)/2}\rho])\o\pi'$ is in the Jacquet module
of $\pi$. Observe that $\d([\nu\rho,\nu^{(b-1)/2}\rho])^2\o\rho\leq
M^*(\delta(\rho,b))$. From this we get that the term $\d( [\nu\rho,\nu^{(a-1)/2}\rho])\t
\d([\nu\rho,\nu^{(b-1)/2}\rho])^2\o\rho\r\pi'$ must show up in
$\mu^*(\d(\rho,b)\r\pi)$. The condition on the cuspidal  support of $\theta$ implies that this term is
a direct summand in the corresponding Jacquet module.
\end{proof}

\begin{corollary}
\label{Cor-3}
Let the assumptions be the same as in the previous lemma. Then the following  requirements on an irreducible subquotient $\tau$ of $\d(\rho,b)\r\pi$ are equivalent:

\begin{enumerate}

\item $\tau$ embeds into $\delta([\nu^{}\rho,\nu^{(b-1)/2}\rho])^2
\t \delta([\nu^{} \rho,\nu^{(a-1)/2}\rho])\r \pi'$.

\item $\tau$ embeds into  $\delta([\nu^{}\rho,\nu^{(b-1)/2}\rho])^2
\t \delta([\nu^{} \rho,\nu^{(a-1)/2}\rho])
\r \lambda$ for some irreducible representation $\lambda$.

\item $\tau$ embeds into $\theta
\r \lambda$ for some irreducible representations $\theta$ and $\lambda$, such that supp$(\theta)=2 [\nu^{}\rho,\nu^{(b-1)/2}\rho]+
[\nu^{} \rho,\nu^{(a-1)/2}\rho].$

\item $\tau$ has $\delta([\nu^{}\rho,\nu^{(b-1)/2}\rho])^2
\t\delta([\nu^{} \rho,\nu^{(a-1)/2}\rho])\o \pi'$ for  a subquotient of its Jacquet module.

\item $\tau$ has $\delta([\nu^{}\rho,\nu^{(b-1)/2}\rho])^2\t\delta([\nu^{} \rho,\nu^{(a-1)/2}\rho])
\o \lambda$ for  a subquotient of its Jacquet module,  for some irreducible representation $\lambda$.

\item $\tau$ has $\theta
\o \lambda$ for  a subquotient of its Jacquet module, for some irreducible representations $\theta$ and $\lambda$, such that supp$(\theta)=2 [\nu^{}\rho,\nu^{(b-1)/2}\rho]+
[\nu^{} \rho,\nu^{(a-1)/2}\rho].$ \qed

\end{enumerate} 

\end{corollary}

\begin{definition}

With the assumptions  as in the previous lemma,  the  irreducible subquotient $\tau$ of $\d(\rho,b)\r\pi$ satisfying the equivalent requirements of the above corollary, will be denoted by
$$
\pi_{\d(\rho,b)}.
$$
The other irreducible subrepresentation of $\d(\rho,b)\r\pi$ will be denoted by
$
\pi_{-\d(\rho,b)}.
$

\end{definition}

\begin{lemma}\label{Le-def-odd-1}
 Let $\pi$ be an irreducible square integrable representation
of a classical group, $ \rho$ an $F'/F$-selfdual irreducible cuspidal representation of a general linear group and 
$$
b\in1+2\mathbb Z_{\geq0}
$$
 such that 
$$
\d(\rho,b)\r\pi
$$
reduces (i.e., $\rho$ and $b$ satisfy (J1),  and $b\not\in\jrp$).

Suppose
 $$
  \jrp=\emptyset.
  $$
   Then $\rho\r\pc$ reduces. Decompose
$$
\rho\r\pc=\tau_1\oplus\tau_{-1}.
$$
Then for an irreducible subrepresentation $T$ of
$$
\d(\rho,b)\r \pi
$$
 there is a unique $i\in\{\pm1\}$ such
that there exists an irreducible representation $\theta$ of a general
linear group such that
$$
T\h \theta\r\tau_{i}.
$$
We shall denote this subrepresentation  $T$ by
$$
\pi_{i\delta(\rho,b)}
$$
(we consider $i\delta(\rho,b)$ as an element of the Hopf algebra $R$ in a natural way).
We can also characterize above $T$ by the requirement that a term of the form
$\theta\o\tau_{i}$ is in the Jacquet module of $T$. 

\end{lemma}

\begin{proof}  Observe
$$
T\h \d(\rho,b)\r \pi\h \d([\rho,\nu^{(b-1)/2}\rho])\t
\d([\nu^{-(b-1)/2}\rho,\nu^{-1}\rho])\r\pi.
$$
We have $\d([\nu^{-(b-1)/2}\rho,\nu^{-1}\rho])\r\pi\cong
\d([\nu\rho,\nu^{(b-1)/2}\rho])\r\pi$
by \eqref{asso}, Lemma \ref{Le-irr} and Proposition 6.1  from \cite{T-RPI}. Thus
$$
T
\h
 \d([\rho,\nu^{(b-1)/2}\rho])\t
\d([\nu\rho,\nu^{(b-1)/2}\rho])\r\pi
\hskip40mm
$$
$$
\cong
\d([\nu\rho,\nu^{(b-1)/2}\rho])\t \d([\rho,\nu^{(b-1)/2}\rho])\r\pi
$$
$$
\hskip40mm
\h
\d([\nu\rho,\nu^{(b-1)/2}\rho])\t
\d([\nu\rho,\nu^{(b-1)/2}\rho])\t\rho\r\pi.
$$
We know that $\pi\h\theta'\r\pc$ for some irreducible
representation
$\theta'$ of a general linear group (this is the definition of $\pc$). Note that the condition
$\jrp=\emptyset$ implies $\rho\t\theta'\cong\theta'\t\rho$.
 Then
$$
T\h \d([\nu\rho,\nu^{(b-1)/2}\rho])\t
\d([\nu\rho,\nu^{(b-1)/2}\rho])\t\rho\t\theta'\r\pc
$$
$$
\cong 
\d([\nu\rho,\nu^{(b-1)/2}\rho])\t
\d([\nu\rho,\nu^{(b-1)/2}\rho])\t\theta'\t\rho\r\pc.
$$

Obviously there exists $i\in\{\pm1\}$ such that
$$
T\h
\d([\nu\rho,\nu^{(b-1)/2}\rho])\t
\d([\nu\rho,\nu^{(b-1)/2}\rho])\t\theta'\r\tau_{i}.
$$
Now Lemma 3.4 implies that there exists an irreducible
representation $\theta$ of a general linear group such that
$$
T\h\theta \r\tau_{i}.
$$
Note that $\rho$ is not in supp$(\theta)$.

Denote by $T'$ the irreducible subrepresentation of $\d(\rho,b)\r\pi$ different from $T$. Suppose that $T'\h \theta' \r\tau_{i}$ for some irreducible $\theta'$. Again, $\rho$ is not in supp$(\theta')$. Then
\begin{equation}
\label{Eq-some1}
\d(\rho,b)\r\pi=T\oplus T'\leq (\theta\oplus\theta')\r \tau_i
\end{equation}
($\rho$ is neither in the cuspidal support of $\theta$, nor of $\theta'$).

The formula for $\mu^*$ implies that the term of form
$*\o\tau_{-i}$ cannot be in the Jacquet module $(\theta\oplus\theta')\r\tau_i$ (use that $\mu^*(\tau_i)=1\o\tau_i+\rho\o\pc$, the fact that $\rho$ is not in the cuspidal supports of $\theta$ and $\theta'$, and use \eqref{Eq-M*} and  \eqref{Eq-jm}).

From the other side, take some irreducible $\theta''\o\pc\leq
\mu^*(\pi)$. Then \eqref{Eq-jm} and \eqref{Eq-M-seg} imply that
\begin{equation}
\label{Eq-some2}
\d([\nu\rho,\nu^{(b-1)/2}\rho])^2\t\theta''\o\rho\r\pc
\end{equation}
is in the Jacquet module of
$\d(\rho,b)\r\pi$. Observe that \eqref{Eq-some1} has a subquotient of the form
$*\o\tau_{-i}$. Now \eqref{Eq-some1} implies that  $(\theta\oplus\theta')\r\tau_i$ has a subquotient of the form $*\o\tau_{-i}$. This contradiction ends the proof.
\end{proof}

Observe that  in the above lemma we have again defined
$$
\pi_{i\delta(\rho,b)}, \quad i\in\{\pm1\}
$$
in this case.

\section{Tempered representations}\label{temp}
\setcounter{equation}{0}

We denote by 
$$
D
$$
 the set of all equivalence classes of irreducible essentially square integrable representations of  $GL(n,F')$, for all $n\geq 1$. The subset of all unitarizable classes in $D$ is denoted by
$$
D^u.
$$

For an irreducible square integrable representation $\pi$ of a classical group denote
$$
D^u_{\pi,\text{red}}=\{\d\in D^u;\d\r\pi \text{ reduces}\}.
$$
Recall that  $\d=\d(\rho,b)\in D^u$ is in $D^u_{\pi,\text{red}}$  if and only if $\rho$ is $F'/F$-selfdual, $\rho$ and $b$ satisfy, (J1) and $(\rho,b)\not\in \jp$. Let
$$
D^u_{\pi,\text{irr}}= D^u \, \backslash \, D^u_{\pi,\text{red}}.
$$

\begin{proposition}
\label{Pr-def-t}
Let $\pi$ be an irreducible square integrable representation of a classical group,
$ \d_1,\dots,\d_n$ different (i.e. nonequivalent) representations in
$D^u_{\pi,\text{red}}$ and $j_1,\dots,j_n\in\{\pm1\}$. Then there exists a unique irreducible subrepresentation $T$ of $\d_1\t\dots\t\d_n\r\pi$ such that 
$$
T\h
\left(\underset{k\in\{1,\dots,n\}\backslash\{i\}}
\t\d_k\right)
\r \pi_{j_i\d_i}
$$
for all $i\in \{1,\dots,n\}.$ We shall denote $T$ by
$$
\pi_{j_1\d_1,\dots,j_n\d_n}.
$$

Further, if we have some other representation $\pi'_{j_1'\d_1',\dots,j_{n'}'\d_{n'}'}$ defined in the above way, then 
$$
\pi_{j_1\d_1,\dots,j_n\d_n}.
\cong
\pi'_{j_1'\d_1',\dots,j_{n'}'\d_{n'}'}
$$
if and only if $\pi\cong \pi'$ and $\{j_1\d_1,\dots,j_n\d_n\}= \{j_1'\d_1',\dots,j_{n'}'\d_{n'}'\}.$

\end{proposition}

\begin{proof} We know that $\d_1\t\dots\t\d_n\r\pi$ is a multiplicity one representation of length $2^n$ (see Theorem 13.1 of \cite{Moe-T}, or \cite{G}).

Now we shall prove by induction that the description of irreducible subrepresentations of $\d_1\t\dots\t\d_n\r\pi$
is well-defined. For $n=1$ there is nothing to prove. Suppose $n\geq 2$ and  that the description is well-defined for $n-1$.  Write $\d_i=\d(\rho_i,a_i)$. Renumerate $(\d_i,j_i)$'s in a way that $a_1\geq a_i$ for all $i>1$. Denote $\tau= \pi_{j_2\d_2,\dots,j_n\d_n}$. 

Observe that  
 \begin{equation}
\label{Eq-main-2}
\aligned
\mu^*(\d_1\t\dots\t\d_n\r\pi)=
\bigg(\prod_{k=1}^n \bigg( \sum_{i_k= -(a_k-1)/2-1}^{ (a_k-1)/2} 
\sum_{j_k=i_k}^{ (a_k-1)/2}
\delta([\nu^{-i_k}\rho,\nu^{(a_k-1)/2}\rho]) 
 \times
 \\
\delta([\nu^{j_k+1} \rho,\nu^{(a_k-1)/2}\rho])
 \otimes
\delta([\nu^{i_k+1} \rho,\nu^{j_k}\rho])
\bigg)
\bigg)\r\mu^*(\pi).
\endaligned
\end{equation}
From this easily follows that the multiplicity of $\d_1\o\tau$ in \eqref{Eq-main-2} is two  (consider the term $\nu^{-(a_1-1)/2}\rho$). 

One gets also directly  that the multiplicity of $\d_1\o\tau$ in $\mu^*(\d_1\r\tau)$ is at least two. The above fact about \eqref{Eq-main-2} implies that the multiplicity is 2.

From this (using Frobenius reciprocity) follows that $\d_1\r\tau$ reduces into two non-equivalent irreducible representations. Denote them by $T_1$ and $T_2$. For description proposed in the proposition, it is enough to prove that both $T_1$ and $T_2$ cannot be in the same time subrepresentations of $\d_2\t\dots\t\d_n\r\pi_{j\d_1}$ for both $j\in \{\pm1\}$. Suppose that they are. Then the multiplicity of $\d_1\o\tau$ in $\mu^*(\d_2\t\dots\t\d_n\r\pi_{j\d_1})$ must be 2.

Observe $1\o \d_i\leq M^*(\d_i)$ for $i=2,\dots,n$, and $\d_1\o\pi\leq \mu^*(\pi_{\pm\d_1})$ (by Frobenius reciprocity). From this follows that  $\d_1\o\tau\leq \mu^*(\d_2\t\dots\t\d_n\r\pi_{\pm\d_1})$. 
From this and the fact that the multiplicity of $\d_1\o\tau$ in \eqref{Eq-main-2} is two, follow that the multiplicity of $\d_1\o\tau$ in  $ \mu^*(\d_2\t\dots\t\d_n\r\pi_{\pm\d_1})$  is one.
This contradiction completes the proof that  representations $\pi_{j_1\d_1,\dots,j_n\d_n}$ are well defined.

The rest of proposition directly follows from Proposition III.4.1 of  \cite{W}.
\end{proof}

\begin{definition}
 If an irreducible (tempered) representation $\pi$ is equivalent to some representation $\pi_{j_1\d_1,\dots,j_n\d_n}$ as in the above proposition, we shall say that $\pi$ is e-tempered representation.
\end{definition}

Note that from \cite{H} follows that in the case of symplectic and split odd-orthogonal groups, the notion of an irreducible e-tempered representation coincides coincides with the notion of elliptic tempered representation.

Now we have:

\begin{theorem} 
\label{Thm-temp-gen}
\begin{enumerate}

\item Let $\pi_{j_1\d_1,\dots,j_n\d_n}$ be an irreducible e-tempered representation (like in the above proposition) and $\gamma_1,\dots,\gamma_m\in D^u_{\pi,\text{irr}}\cup\{\d_1,\dots,\d_n\}$. Then
$$
\gamma_1\t\dots\t\gamma_m\r \pi_{j_1\d_1,\dots,j_n\d_n}
$$
is an irreducible tempered representation.

\item If we have additionally an irreducible e-tempered representation $ \pi'_{j_1'\d_1',\dots,j_{n'}'\d_{n'}'}$   and $\gamma_1',\dots,\gamma_{m'}'\in D^u_{\pi',\text{irr}}\cup\{\d_1',\dots,\d_{n'}'\}$. Then
$$
\gamma_1\t\dots\t\gamma_m\r \pi_{j_1\d_1,\dots,j_n\d_n}
\cong
\gamma_1'\t\dots\t\gamma_{m'}'\r \pi_{j_1'\d_1',\dots,j_{n'}'\d_{n'}'}
$$
if and only if we have equality of multisets 
$$
(\gamma_1,\dots,\gamma_m, \check\gamma_1,\dots,\check\gamma_m)
=
(\gamma_1',\dots,\gamma_{m'}',\check \gamma_1',\dots,\check\gamma_{m'}'),
$$
 and $ \pi_{j_1\d_1,\dots,j_n\d_n}
\cong
 \pi_{j_1'\d_1',\dots,j_n'\d_{n'}'}$ (see the previous proposition for the description this equivalence).
 
 \item Each irreducible tempered representation of a classical group $S_q$ is equivalent to some representation from (1).
 \qed
\end{enumerate}
\end{theorem}

Now we shall define in a natural way Jordan blocks of irreducible tempered representations of classical groups. We shall  use here rather irreducible square integrable representations, than the pairs that parameterize them. Therefore, we
   define
$$
Jord(\pi)_{\text{d.s.}}= \{\delta(\rho,n); (\rho,n)\in Jord(\pi)\}.
$$

We can  consider a set  in an obvious way as a multiset (we shall do this below).

We now extend the definition of Jordan blocks to the case of irreducible tempered representations. This is very simple and natural extension, which seems to be present in the literature for a long time (at least implicitly). 

 \begin{definition}
 Let $\pi_{j_1\d_1,\dots,j_n\d_n}$ be an irreducible e-tempered representation defined in Proposition \ref{Pr-def-t}, and let  
 $$
 \gamma_1,\dots,\gamma_m\in D^u_{\pi,\text{irr}}\cup\{\d_1,\dots,\d_n\}. 
 $$
 Then the Jordan blocks 
 $$
 Jord(\gamma_1\t\dots\t\gamma_m\r \pi_{j_1\d_1,\dots,j_n\d_n}) 
 $$
  attached to the irreducible tempered representation
$
\gamma_1\t\dots\t\gamma_m\r \pi_{j_1\d_1,\dots,j_n\d_n}
$
are defined to be the multiset
$$
(\gamma_1,\dots,\gamma_m, \check\gamma_1,\dots,\check\gamma_m) +2(\d_1,\dots,\d_n)+
Jord(\pi)_{\text{d.s.}}.
$$

\end{definition}

\bigskip

We  attach to $Jord(\gamma_1\t\dots\t\gamma_m\r \pi_{j_1\d_1,\dots,j_n\d_n}) $ the corresponding set
$$
\{\gamma_1,\dots,\gamma_m, \check\gamma_1,\dots,\check\gamma_m\} \cup\{\d_1,\dots,\d_n\}\cup Jord(\pi)_{\text{d.s.}},
$$
which will be denoted by
$$
|Jord(\gamma_1\t\dots\t\gamma_m\r \pi_{j_1\d_1,\dots,j_n\d_n}) |.
$$

We can now extend the definition of partially defined function to the tempered case:

 \begin{definition}
 Let $\pi_{j_1\d_1,\dots,j_n\d_n}$ be an irreducible e-tempered representation  defined in Proposition \ref{Pr-def-t}, and let $\gamma_1,\dots,\gamma_m\in D^u_{\pi,\text{irr}}\cup\{\d_1,\dots,\d_n\}$. Denote by $\tau$ 
  the irreducible tempered representation
  $$
  \tau=\gamma_1\t\dots\t\gamma_m\r \pi_{j_1\d_1,\dots,j_n\d_n} .
  $$ 
Then the partially defined function $\epsilon_\tau$ attached to $\tau$ is a function defined on a subset 
of $|Jord(\tau)|\cup |Jord(\tau)|\t |Jord(\tau)|$ (with values in $\{\pm1\}$), which satisfies the following requirements:

\begin{enumerate}
\item If $\d\in |Jord(\tau)|$ is not $F'/F$-selfdual, then $\e_\tau(\d)$ is not defined.

\item If $\d\in |Jord(\tau)|$ is  $F'/F$-selfdual, if it has even multiplicity in $Jord(\tau)$ and if $\d\not\in\{\d_1,\dots,\d_n\}$\footnote{In this case, if we write $\d$ as $\d(\rho,a)$, then $(\rho,a)$ does not satisfy (J1)}, then $\e_\tau(\d)$ is not defined.

\item If $\d\in |Jord(\tau)|$ is  $F'/F$-selfdual, if it has even multiplicity in $Jord(\tau)$ and if $\d\in\{\d_1,\dots,\d_n\}$, then $\e_\tau(\d)$ is  defined, and
$$
\e_\tau(\d_i)=j_i.
$$

\item  If $\d=\d(\rho,a)\in |Jord(\tau)|$ has odd multiplicity in $Jord(\tau)$, then $\d\in Jord(\pi)_{\text{d.s.}}$, and further $\e_\tau(\d)$ is  defined if and only if $\e_\pi((\rho,a))$ is defined. If it is defined, then
$$
\e_\tau(\d)=\e_\pi((\rho,a)).
$$

\item  Let $\d_1=\d(\rho_1,a_1),\d_2=\d(\rho_2,a_2)\in |Jord(\tau)|$. If $\e_\tau((\d_1,\d_2))$ is defined, then $\d_1,\d_2\in Jord(\pi)_{\text{d.s.}}$. 
In that case $\e_\tau((\d_1,\d_2))$ is  defined if and only if $\e_\pi((\rho_1,a_1))e_\pi((\rho_2,a_2))^{-1}$ is defined,
and then
$$
\e_\tau((\d_1,\d_2))=\e_\pi((\rho_1,a_1))e_\pi((\rho_2,a_2))^{-1}.
$$

\end{enumerate}

\end{definition}

Now we shall give very brief description of the irreducible tempered representations  in terms of Jordan blocks,  partially defined functions and partial cuspidal supports (as it was done for irreducible square integrable representations in Theorem 6.1 of \cite{Moe-T}). 

\begin{definition} A tempered triple
$$
(Jord,\s,\e)
$$
is a triple  for which holds:

\begin{enumerate}

\item $\text{Jord}$ is a finite multiset in $D^u$ which satisfies:

\begin{enumerate}

\item
Jord is $F'/F$-selfdual (i.e. if  $\text{Jord}=(\d_1,\dots,\d_k)$, then $\text{Jord}=(\check\d_1,\dots,\check\d_k)$). 
 
 \item 
 If $\d=\d(\rho,a)$ from Jord is $F'/F$-selfdual and $(\rho,a)$ does not satisfy (J1), then the multiplicity of $\d$ in Jord is even.
 
 \end{enumerate}

\item $\s$ is an irreducible cuspidal representation of a classical group.

\item
$\e$ is a function which takes values in $\{\pm1\}$. It is defined on a subset of the set $|\text{Jord}|\cup |\text{Jord}|\t |\text{Jord}|$. Each $\d_1$ or $(d_2,\d_3)$ from the domain of $\e$ satisfies:
$$
\text{ \ \ \ \  \ \ \ \ If $\d=\d(\rho,a)=\d_1$, $\d_2$ or $\d_3$, then $\d$ is $F'/F$-selfdual and $(\rho,a)$ satisfies (J1).}
$$
The domain of $\e$ is described as follows: 

\
\begin{enumerate} 

\item If the multiplicity of   $\d$ in $\text{Jord}$ is even, then $\e$ is defined on $\d$. 

\item If the multiplicity of   $\d$ in $\text{Jord}$ is odd, then $\e$   is defined on $\d=\d(\rho,a)$ if and only if $a$ is even or $a$ is odd and $\text{Jord}_\rho(\sigma)=\emptyset$.

\item $\e$ is defined on $(\d(\rho_1,a_1),\d(\rho_2,a_2))$ if and only if the multipicities of both $\d(\rho_1,a_1)$ and $\d(\rho_2,a_2)$ in Jord are odd, and if $\rho_1=\rho_2$.

\end{enumerate} 

\

The partially defined function $\e$ needs to satisfy: 

\

\begin{enumerate} 
\item[(a)]
If $\e$ is defined on $\d_1$, $\d_2$ and $(\d_1,d_2)$, then
$$
\e(\d_1,\d_2)=\e(d_1)\e(d_2)^{-1}.
$$

\item[(b)]
If $\e$ is define on $(d_1,\d_2)$ and $(d_2,\d_3)$, then 
$$
\e(d_1,\d_3)=\e(d_1,\d_2)\e(d_2,\d_3).
$$

\end{enumerate} 

\item Denote by $\text{Jord}^{\text{(J1),odd}}$ the multiset that we get  from Jord deleting all $\d=\d(\rho,a)$'s which are not $F'/F$-selfdual, or which are, but 
 have even multiplicity in Jord (all what remains satisfy (J1)). Let $\e'$ be the natural restriction of $\e$ to $|\text{Jord}^{\text{(J1),odd}}|\cup |\text{Jord}^{\text{(J1),odd}}|\t |\text{Jord}^{\text{(J1),odd}}|$. Denote $\{(\rho,a);\d(\rho,a)\in |\text{Jord}^{\text{(J1),odd}}|\}$ by $|\text{Jord}^{\text{(J1),odd}}|'$, and denote by $\e''$ the function that we get on $|\text{Jord}^{\text{(J1),odd}}|$ transferring $\e'$ in obvious way. Then
$$
(|\text{Jord}^{\text{(J1),odd}}|',\s,\e'')
$$
needs to be an admissible triple (as it is defined in \cite{Moe-Ex}; see also \cite{Moe-T}).
\end{enumerate}
\end{definition}

 Now we can express the parameterization of the irreducibe  tempered representations that we have obtained in the following way. 
The map 
\begin{equation}
\tau \rightarrow (Jord(\tau),\tau_{cusp},\epsilon_\tau)
\end{equation}
 defines a bijection from the
set of  equivalence classes of irreducible tempered
representations of groups $S_n$  onto the set of all tempered triples.

\begin{remark}
We   consider  quasi-split classical groups
and their generic irreducible tempered representations in this remark. For these groups, we shall follow the  conventions of \cite{Han} regarding non-degenerate characters of the maximal unipotent subgroups. 

 The definition of the partially defined function  attached to an irreducible square integrable representation of a classical group depends (only) on the choice of  indexing of irreducible  constituents of
$$
\rho\rtimes \pi_{cusp}=\tau_{1}\oplus\tau_{-1},
$$
when $(\pi_{cusp}, \rho)$ runs over all pairs of irreducible $F'/F$-selfdual cuspidal representations $\rho$ of   general linear groups and  irreducible cuspidal representations  $\pi_{cusp}$ of classical groups, such that $\rho\rtimes \pi_{cusp}$ reduces (see \cite{T-inv} for more details). In this remark, we shall  assume that in the case that $ \pi_{cusp}$ is generic,  we have chosen always $\tau_1$ to be generic. Then from \cite{Han} and \cite{Mu2} follows directly that an irreducible square integrable representation $\pi$ of a quasi-split classical group is generic (for the fixed non-degenerate character of the maximal unipotent subgroup) if and only if the partial cuspidal support $ \pi_{cusp}$ of $\pi$ is generic and if the partially defined function $\epsilon_\pi$ attached to $\pi$  takes value one on elements and pairs from $Jord(\pi)$, whenever it is defined on them (see also \cite{Mu1}). This was also proved by C. M\oe glin in an unpublished manuscript. We shall briefly comment this result.

If $\pi$ is a generic irreducible square integrable representation, then obviously (by Theorem 2 of \cite{R-Wh}) the partial cuspidal support $ \pi_{cusp}$ of $\pi$ is generic. Further,
in Proposition 3.1 of \cite{Han} is shown that  the partially defined function $\epsilon_\pi$ attached to $\pi$  takes value one on  pairs from $Jord(\pi)$, whenever it is defined on them (see also \cite{Mu1}).
Moreover, by  the remark after Proposition 3.1 of \cite{Han}, the partially defined function takes value one on elements  from $Jord(\pi)$, whenever it is defined on them (one can consult   also \cite{Mu1}  again). Let $(\rho,a)\in Jord(\pi)$. 
We shall additionally comment here the case of odd $a$. Then $\rho\rtimes \pi_{cusp}$ reduces. Now Proposition 4.1 of \cite{T-inv} and \cite{R-Wh}  (and also Proposition 3.1 of \cite{Han}) imply that the partially defined function  on the element   $(\rho,a)$ takes the value one.

Let $\pi$ be an irreducible square integrable representation  of a classical group such that the partial cuspidal support $ \pi_{cusp}$ of $\pi$ is generic, and such that the partially defined function $\epsilon_\pi$ attached to $\pi$  takes value one on elements and pairs from $Jord(\pi)$, whenever it is defined on them.
Now first Proposition 1.1 of \cite{Mu2} implies that the generic subquotient of the representation parabolically induced  by corresponding irreducible cuspidal representation, where $\pi$ is a subquotient, must be square integrable. Now the above discussion implies that this generic square integrable subquotient is $\pi$. Thus, $\pi$ is generic.

Let $\pi$ be an irreducible generic square integrable representation  of a classical group and let $\delta$ be an  irreducible $F'/F$-selfdual square integrable representation of  a general linear group, such that $\delta\rtimes \pi$ reduces. Then  precisely one of the irreducible constituents of $\delta\rtimes \pi$ is generic (\cite{R-Wh}). Denote it by $\tau_{gen}$. Let $\d=\d(\rho,b)$, where $\rho $ is irreducible $F'/F$-selfdual cuspidal representation of  a general linear group and $b\geq 1$.

Suppose that we are in the setting of  Lemma 
\ref{Le-def-main}.
The representation 
$$
\delta([\nu^{(a-1)/2+1}\rho,\nu^{(b-1)/2}\rho])^2
\t \delta([\nu^{-(a-1)/2} \rho,\nu^{(a-1)/2}\rho])\r \pi
$$
 has a unique generic irreducible  subquotient, and it must be equivalent to $\tau_{gen}$ (\cite{R-Wh}). By the generalized injectivity conjecture proved for classical groups in \cite{Han}, the generic subquotient must be a subrepresentation. Thus, 
$$
\tau_{gen} \hookrightarrow 
\delta([\nu^{(a-1)/2+1}\rho,\nu^{(b-1)/2}\rho])^2
\t \delta([\nu^{-(a-1)/2} \rho,\nu^{(a-1)/2}\rho])\r \pi.
$$
Now    Definition
\ref{Def-1} (see also Corollary 
\ref{Cor-1})
implies $\pi_{\delta(\rho,b)}\cong \tau_{gen}$. Thus, $\pi_{\delta(\rho,b)}$ is generic (and $\pi_{-\delta(\rho,b)}$ is not).
In the same way one sees that also in the settings of Lemma 
\ref{Le-def-even}
and Lemma 
\ref{Le-def-odd-2},
 $\pi_{\delta(\rho,b)}$ is also generic and $\pi_{-\delta(\rho,b)}$ is  not.
 In the setting of 
 \ref{Le-def-odd-1},
 \cite{R-Wh} implies that $\pi_{\delta(\rho,b)}$ is  generic.

 This implies that $\pi_{\epsilon\delta}$ is generic if and only if $\pi$ is generic and $\epsilon=1$. 

From this and directly from \cite{R-Wh} follows that an irreducible $e$-tempered representation 
$$
\pi_{\epsilon_1\delta_1,\dots ,\epsilon_n\delta_n}
$$
 is generic if and only if $\pi$ is generic and $\epsilon_1=\dots=\epsilon_n=1$.

At the end we conclude (from above discussion and \cite{R-Wh}) that the irreducible tempered representation in (1) of Theorem \ref{Thm-temp-gen} is generic, if and only if the $e$-tempered representation there is generic.

\end{remark}

\section{Some elementary facts on Bernstein center of direct products}\label{elem-b-c}

\setcounter{equation}{0}

First we shall recall  the definition of the Bernstein center (see \cite{B} and \cite{B-D} for more details).

Let $G$ be the group of $F$-rational points of a connected reductive group defined over $F$ (we have fixed a local non-archimedean field $F$).
The set of equivalence classes of irreducible smooth representations of $G$ is called the non-unitary dual of $G$ and denoted by $\tilde G$.  It carries a natural topology of uniform convergence of matrix coefficients over compact subsets (see \cite{T-Geo} for  more details). Denote the Hausdorffization of $\tilde G$ by $\Theta (G)$. Then $\Theta (G)$ is the set of conjugacy classes  of all pairs $(M,\rho)$, where $M$ is a Levi subgroup of $G$ and $\rho$ is an equivalence class of  irreducible cuspidal representations of $M$ (the conjugacy class of $(M,\rho)$ will be denoted by $[(M,\rho)]$). The fibers of the Hausdorffization map $
\tilde G\to \Theta(G)$ are finite. The structure of the complex algebraic variety on the unramified characters of Levi subgroups defines in a natural way a structure of algebraic variety on $\Theta (G)$. The algebra of all the regular functions on $\Theta (G)$ will be denoted by $\mathfrak Z(G)$. The subalgebra of all regular functions supported only on finitely many connected components is denoted by $\mathfrak Z(G)_0$.

The Bernstein center of $G$ is the algebra $\mathcal Z(G)$ of all endomorphisms of the category of all smooth representations of $G$. For such an endomorphism $z$, denote by $\tilde z:\tilde G\to \mathbb C$ the mapping which attaches to $\pi\in \mathbb C$ the scalar $\chi_\pi(z)$ by which $z$ acts in the representation space of $\pi$. Now $\tilde z$ factors through the Hausdorffization $\tilde G\to \Theta(G)$, and it is a regular function on $\Theta(G)$, i.e. an element of $\mathfrak Z(G)$. In this way one gets an isomorphism of the Bernstein center $\mathcal Z(G)$ onto the algebra $\mathfrak Z(G)$. In what follows, we shall identify the Bernstein center $\mathcal Z(G)$ with $\mathfrak Z(G)$.

Let  $\chi$ be a character of $\mathfrak Z(G)$ which does not vanish on $\mathfrak Z(G)_0$. In this paper we,  shall consider only such characters of $\mathfrak Z(G)$, and call them infinitesimal characters of $G$. For such a character, there exists $[(M,\rho)]\in \Theta(G)$ such that $\chi$ is the evaluation in that point.

Let $(\pi,V)$ be a smooth representation  $G$. Then $\pi$ is called a $\chi$-representation, or a representation with infinitesimal character $\chi$, if each $z\in \mathfrak Z(G)$ acts in the representation space of $\pi$ as the multiplication by the scalar $\chi(z)\in \mathbb C$. Further, $\pi$ will be called a representation of type $\chi$ if each irreducible subquotient of $\pi$ is a $\chi$-representation.

For  a smooth representation $(\pi,V)$ of  $G$ we denote by  
$$
(\pi_{[\chi]},V_{[\chi]})
$$
 the biggest subrepresentation of $(\pi,V)$ which is of type $\chi$ (one easily shows that such subrepresentation exists). Note that the possibility that this subrepresentation is $\{0\}$ is not excluded (actually, in the cases that we shall consider, this will be almost always the case). Denote
 $$
 V_\chi=
\{v\in V; z.v=\chi(z)v, \forall z\in \mathfrak Z(G)\}.
 $$
 This is a subrepresentation of $\pi$,  denoted by $\pi_\chi$. Clearly, this is a $\chi$-representation.

\smallskip

Let $H$ and $L$ be the  groups of $F$-rational points of a connected reductive group defined over $F$. Denote
$$
M=H\times L.
$$
Then
$$
(\sigma,\tau)\mapsto \sigma\otimes\tau
$$
gives identification of $\tilde H\times\tilde L$ with $\tilde M$. In what follows, we shall always assume this identification.

In the rest of this section we shall consider representations of $M$, and the action of the Bernstein center of $H$ on the representations of $M$. 

\begin{lemma}\label{Le-bc}
\begin{enumerate}
\item
Let $\chi$ be an infinitesimal character of $H$, let $\sigma$  be a smooth $\chi $-repre\-sentation of $H$ and let $\tau$ be a smooth representation of $L$. Then    the representation $\sigma\otimes\tau$ of $M$ is a $\chi$-representation, when we look at it as a representation of $H$ (by restriction).

\item Let $(\mu,V)$ be a smooth representation of $M$ of finite length. Then there exists finitely many infinitesimal characters $\chi_1,\dots,\chi_k$ of $H$ such that
\begin{equation}\label{DirSum}
V=V_{[\chi_1]}\oplus\dots\oplus V_{[\chi_k]}.
\end{equation}
Moreover, each $V_{[\chi_i]}$ and each $V_{\chi_i}$ is an $M$-subrepresentation of $V$.
\end{enumerate}
\end{lemma}

\begin{proof} The claim (1) is evident.

For (2), recall that by the definition, $V_{[\chi_i]}$ is invariant for the action of $H$. From the other side, the action of $L$ commutes with the action of $H$ (i.e. the operators corresponding to the action of $L$ can be regarded as an  $H$-intertwinings). This implies that the action of $\mathfrak Z(H)$ commutes with the action of $M$. Therefore, if we act by $M$ on $V_{[\chi_i]}$, we get again representation of type $\chi_i$. The maximality of $V_{[\chi_i]}$ implies $M$-invariance of it.
Further, $V_{\chi_i}$ is an $M$-subrepresentation since the action of $L$ commutes with the action of $H$.

Let  $\sigma_1\otimes \tau_1$, $\sigma_2\otimes \tau_2$,    \dots\dots  , $\sigma_n\otimes \tau_n$ be the set of all  equivalence classes of irreducible subquotients of $\mu$. Write $\{\chi_{\sigma_1},\dots,\chi_{\sigma_n}\}=\{\chi_1,\dots,\chi_k\}$, where we assume $\chi_i\ne\chi_j$ for $i\ne j$.  Then each $\chi_i$ is evaluation at some point $\theta_i\in\Theta(M)$.  Now we shall prove that \eqref{DirSum} holds (the proof is standard, but nevertheless we give details below). We prove by induction on $k$ that the sum $V_{[\chi_1]}+\dots+ V_{[\chi_k]}$ is direct. Suppose $k=2$ and chose  $z\in \mathfrak Z(G)$ such that $z(\theta_{1})=0$ and $z(\theta_{2})=1$. Let $v\in V_{[\chi_1]}\cap V_{[\chi_2]}$. Suppose $v\ne 0$. Now since $v\in V_{[\chi_1]}$, we  can chose integer $e\geq 0$ such that $z^e. v\ne 0$  and $z^{e+1}.v=0$ (for $e=0 $ we take formally $z^0.v=v$). Denote $v'=z^e. v$. Then $v'\ne 0$ and $z. v'=0$. But condition $v'\in V_{[\chi_2]}$ and $v'\ne 0$ easily implies $z.v'\ne0$ (using $z(\theta_2)=1$). This contradiction shows that the sum $V_{[\chi_1]}+ V_{[\chi_2]}$ is direct.

Now we shall present the inductive step. Let $k\geq 3$. Fix $i_0\in\{1,\dots,k\}$.  Take $z_0\in \mathfrak Z(G)$ such that $z_0(\theta_{i_0})=1$ and $z_0(\theta_{i})=0$ for $i\ne i_0$. Take
$$
v\in V_{[\chi_{i_0}]}\cap (\oplus_{i\ne i_0}V_{[\chi_i]}).
$$
Suppose $v\ne 0$. Now since $v\in V_{[\chi_{i_0}]}$,  for any integer $\ell\geq 1$ the fact that $z_0^\ell(\theta_{i_0})=1$ easily implies 
\begin{equation}
\label{ne0}
z_0^\ell.v\ne0
\end{equation}
(look at the subrepresentation generated by $v$, and some its maximal subrepresentation).

 We can write $v'=\sum_{i\ne i_0} v_i'$ where $v_i'\in  V_{[\chi_i]}$. Now this formula and the fact that $z_0(\theta_i)=0$ for $i\ne i_0$ imply that 
 $$
 z_0^m.v=0
 $$
  for some (big enough) integer $m\geq1$. This contradicts \eqref{ne0}.

Suppose $V_{[\chi_1]}+\dots+ V_{[\chi_k]}\ne V$. Chose an $M$-subrepresentation $V'$ of $V$ containing $V_{[\chi_1]}+\dots+ V_{[\chi_k]}$, such that the quotient representation $V'/(V_{[\chi_1]}+\dots+ V_{[\chi_k]})$ is irreducible. Then the infinitesimal character of this quotient must be some $\chi_{i_0}$. 
Take again some $z_0\in \mathfrak Z(G)$ such that $z_0(\theta_{i_0})=1$ and $z_0(\theta_{i})=0$ for $i\ne i_0$.  
Now we can chose integer $l\geq 1$ such that  holds
$$
z^l_0.w\in V_{[\chi_{i_0}]}
$$
for any  $w\in V_{[\chi_1]}+\dots+ V_{[\chi_k]}$.
 Let $v\in V'\backslash  (V_{[\chi_1]}+\dots+ V_{[\chi_k]})$. Then 
 $$
 z^m_0.v-v\in V_{[\chi_1]}+\dots+ V_{[\chi_k]}
 $$
    for any integer $m\geq 1$, which implies $z_0^m.v\in V'\backslash  (V_{[\chi_1]}+\dots+ V_{[\chi_k]})$ since $z_0(\theta_{i_0})=1$.
Now for any $z\in\mathfrak Z(H)$ holds
$$
z.(v+(V_{[\chi_1]}+\dots+ V_{[\chi_k]})=\chi_{i_0}(z)v+(V_{[\chi_1]}+\dots+ V_{[\chi_k]}),
$$
which implies
$$
z.v-\chi_{i_0}(z)v\in V_{[\chi_1]}+\dots+ V_{[\chi_k]},
$$
and further
\begin{equation}\label{Eq-chi}
z.(z^l_0.v)-\chi_{i_0}(z)(z^l_0.v )\in V_{[\chi_{i_0}]}.
\end{equation}
Recall $z^l_0.v\notin  V_{[\chi_{i_0}]}$ since $z^m_0.v-v\in V_{[\chi_1]}+\dots+ V_{[\chi_k]}$ for any integer $m\geq 1$, and $v\notin  V_{[\chi_1]}+\dots+ V_{[\chi_k]}$. Now \eqref{Eq-chi} implies that $(V'/V_{[\chi_{i_0}]})_\chi$ is non-trivial. Consider the projection $p:V'\to V'/V_{[\chi_{i_0}]}$. Then $p^{-1}((V'/V_{[\chi_{i_0}]})_\chi)$ is a representation of type $\chi_{i_0}$, and it strictly contains $V_{[\chi_{i_0}]}$. This contradiction completes the proof of the lemma.
\end{proof}

\begin{corollary}\label{Cor-bc} Let $G$ be the group of $F$-rational points of a connected reductive group defined over $F$ and let $P=MN$ be its Levi subgroup. Suppose that $M$ is (isomorphic to) a direct product of reductive groups $H$ and $L$. Take an irreducible smooth representation $(\pi,V)$ of $G$ such that the normalized Jacquet module $r^G_M(\pi)$ of $\pi$ with respect to $P$ is non-trivial (i.e. $\ne\{0\}$). Let $(\sigma\otimes \tau,U)$  be an   irreducible subquotient of  $r^G_M(\pi)$. Then

\begin{enumerate}
\item
There exists an irreducible smooth representation $\sigma'\otimes \tau'$ of $M$ such that
$$
\chi_{\sigma'}=\chi_\sigma
$$
and
$$
\pi\hookrightarrow \text{\rm Ind}_P^G(\sigma'\otimes \tau').
$$

\item
If for any irreducible subquotient $\mu$ of $(r^G_M(\pi))_{[\chi_\sigma]}$ there exists an irreducible representation $\tau'$ of $L$ such that  $\mu\cong\sigma\otimes \tau'$, then  
$$
\pi\hookrightarrow \text{\rm Ind}_P^G(\sigma\otimes \tau'')
$$
for some irreducible representation $\tau''$ of $L$.

\item If $\sigma$ is cuspidal, then there exists an irreducible smooth representation $ \tau'$ of $L$ such that
$$
\pi\hookrightarrow \text{\rm Ind}_P^G(\sigma\otimes \tau').
$$

\item If  $(r^G_M(\pi))_{[\chi_\sigma]}$ is an irreducible $M$-representation, then
$$
\pi\hookrightarrow \text{\rm Ind}_P^G(\sigma\otimes
\tau
).
$$
More generally, if all irreducible subquotients of $(r^G_M(\pi))_{[\chi_\sigma]}$ are isomorphic to $\sigma\otimes \tau$, 
then again $\pi\hookrightarrow \text{\rm Ind}_P^G(\sigma\otimes \tau).$
\end{enumerate}
\end{corollary} 

\begin{proof}  (1) We know by (1) of Lemma  \ref{Le-bc} that $(r_M^G(\pi))_{[\chi_\sigma]}\ne \{0\}$. By (2) of the same lemma, $(r_M^G(\pi))_{[\chi_\sigma]}$ is $M$-invariant. Since it is of finite length, it has an irreducible quotient. Denote it by $\sigma'\otimes\tau'$. Frobenius reciprocity implies
$
\pi\hookrightarrow \text{\rm Ind}_P^G(\sigma'\otimes \tau').
$
Observe that $\sigma'\otimes\tau'$ is a $\chi_{\sigma'}$-representations (as $H$-representation). But irreducible quotient of $(r_M^G(\pi))_{[\chi_\sigma]} $ has infinitesimal character $\chi_\sigma$ as $H$-representation. Therefore, $\chi_{\sigma'}=\chi_\sigma$. The proof of the (1) is now complete.

(2) is immediate consequence of (1).

(3) If $\sigma$ is cuspidal and $\sigma'$ is irreducible such that $\chi_\sigma=\chi_{\sigma'}$, then $\s\cong\s'$. Now from (2) immediately follows (3).

(4) follows also directly from (1).
\end{proof}

\begin{remark}\label{GL-inf} 
Two irreducible representations of general linear groups have the same infinitesimal character if and only if they have the same cuspidal support.
\end{remark}

\section{Interpretation of invariants of square integrable representations in terms of Jacquet modules}\label{jac-mod-int}
\setcounter{equation}{0}

Recall that an irreducible cuspidal representation $\s$ of a classical group (from the same series) is called the partial cuspidal support of an irreducible square integrable representation $\pi$ of a classical group, if there exists a smooth representations $\theta$ of a general linear group  such that
$
\pi\h\theta\r\s.
$
Now  
Frobenius reciprocity implies that  
$\theta\o\s$ is a quotient  of the appropriate Jacquet module of $\pi$. 
 In particular, $\theta\o\s$ is a subquotient, i.e.
$$
\theta\o\s\leq \mu^*(\pi).
$$
The following well-known proposition (which directly follows from (3) of Corollary \ref{Cor-bc}) implies that the converse holds:

\begin{proposition}\label{Pr-pcs} Let $\pi$ be an irreducible square integrable representation  of a classical group. Suppose that there exits  an irreducible cuspidal representation $\sigma$ of a classical group and an irreducible representations $\theta$ of a general linear group such that
$$
\theta\o\s\leq \mu^*(\pi).
$$
Then $\sigma$ is the partial cuspidal support of $\pi$.
\qed

\end{proposition}

  Let $\pi$ be an irreducible square integrable representation of a
classical group $S_q$. Suppose that we have $a\in\jrp$ which has $a_-\in \jrp$. Suppose that
$$
\e_\pi((\rho,a))\e_\pi((\rho,a_-))^{-1}=1.
$$
This is the case if and only if  there exists an irreducible representation $\sigma$ of classical group such that
$$
\pi\h \delta([ \nu^{(a_--1)/2+1} \rho, 
\nu^{(a-1)/2}\rho]\r \s.
$$
This and Frobenius reciprocity imply that $\delta([\nu^{(a_--1)/2+1}\rho, \nu^{(a-1)/2}\rho])\o\sigma$  is a quotient of the Jacquet module of $\pi$. In particular,
$$
\delta([\nu^{(a_--1)/2+1}\rho, \nu^{(a-1)/2}\rho])\o\sigma\leq \mu^*(\pi).
$$
We have the converse:

\begin{proposition}\label{Pr-eps=} Let $\pi$ be an irreducible square integrable representation of a
classical group $S_q$. Suppose that there exists $a\in\jrp$ which has $a_-\in \jrp$, such that
$$
\delta([\nu^{(a_--1)/2+1}\rho, \nu^{(a-1)/2}\rho])\o\sigma\leq \mu^*(\pi)
$$
for some irreducible representation $\s$.  Then
$$
\e_\pi((\rho,a))\e_\pi((\rho,a_-))^{-1}=1.
$$
\end{proposition}

\begin{proof} 
The condition on the Jacquet module in the proposition and transitivity of Jacquet modules imply that
$$
\nu^{(a-1)/2}\rho\o\nu^{(a-1)/2-1}\rho\o\dots\o\nu^{(a_--1)/2+1}\rho\o
\s
$$
is a subquotient of the Jacquet module of $\pi$. Now (3) of Corollary \ref{Cor-bc} implies that 
\begin{equation*}
\pi\h \nu^{(a-1)/2}\rho\t\nu^{(a-1)/2-1}\rho\t\dots\t\nu^{(a_--1)/2+1}\rho\r
\sigma'
\end{equation*}
for some irreducible representation $\sigma'$. 
Further, (1) of Corollary \ref{Cor-bc} and Remark \ref{GL-inf} imply that 
$$
\pi\h \theta\r\sigma''
$$
for some irreducible subquotient $\theta$ of $\nu^{(a-1)/2}\rho\t\nu^{(a-1)/2-1}\rho\t\dots\t\nu^{(a_--1)/2+1}$ and some irreducible representation $\sigma''$.  
Denote
$$
k_i=(a-1)/2-i+1, \quad i=1,\dots , (a-a_-)/2.
$$
 Then 
\begin{equation}\label{Eq-index}
\theta\h \nu^{k_{\a(1)}}\rho \t\dots\times\nu^{k_{\alpha({(a-a_-)/2})}}\rho
\end{equation}
for some permutation $\a$ of $\{1,\dots,(a-a_-)/2\}$.
This and 
$\pi\h \theta\r\sigma''$ imply
$$
\pi\h \nu^{k_{\a(1)}}\rho\t\dots\times\nu^{k_{\alpha({(a-a_-)/2})}}\rho\r\sigma''.
$$
Let
$$
j=\max\{1\leq i\leq (a-a_-)/2+1;  \alpha(s)=s \text{ for }   s=1,\dots ,i-1\}.
$$
If $j=(a-a_-)/2+1$, then \eqref{Eq-index} implies $\theta\cong \delta([\nu^{(a_--1)/2+1}\rho, \nu^{(a-1)/2}\rho])$ and now $\pi\h \theta\r\sigma''$ implies
$$
\epsilon_{\pi}((\rho,a_-))\epsilon_{\pi}((\rho,a))^{-1}=1.
$$
Suppose 
$$
j\leq(a-a_-)/2.
$$
 Then $j<\a(j)$ and  in \eqref{Eq-index} we can bring $\nu^{k_{\a(j)}}\rho$ at the "beginning", i.e. $\theta\h \nu^{k_{\alpha(j)}}\rho\t\dots $. Now Frobenius reciprocity and (1) of Proposition \ref{Pr-jb-jac} imply
 $$
 2k_{\a(j)}+1\in\jrp.
 $$
  Note that
 $$
 2k_{\a(j)}+1=2((a-1)/2-{\a(j)}+1)+1=a-2\a(j)+2.
 $$
 From $-(a-a_-)/2\leq -\a(j)<-j$ we get 
 $$
a_-+2=a-(a-a_-)+2 \leq a-2\a(j)+2 (=2k_{\a(j)}+1)<a-2j-2\leq a.
 $$
 Thus $a_-< 2k_{\a(j)}+1<a$.
Since $2k_{\a(j)}+1\in\jrp$,  this contradicts to the assumption that $a$ and $a_-$ are the neighbors in $\jrp$.
\end{proof}

\begin{remark}\label{epsilon=}  We shall consider the situation as in previous proposition, i.e. $\pi$ is an irreducible square integrable representation of a
classical group $S_q$, $a\in\jrp$ which has $a_-\in \jrp$, such that
$$
\e_\pi((\rho,a))=\e_\pi((\rho,a_-)).
$$
Now  Lemma 5.1 in \cite{Moe-T} (of C. M\oe glin) implies
$$
\pi\h \delta([ \nu^{-(a_--1)/2} \rho, 
\nu^{(a-1)/2}\rho])\r \pi'
$$
for an irreducible square integrable representation $\pi'$. By Proposition \ref{Pr-jb-main}, we know $$
Jord(\pi')=\jrp\backslash \{(\rho,a),\rho,a_-)\}.
$$
 This implies that
$$
 \delta([ \nu^{-(a_--1)/2} \rho, 
\nu^{(a_--1)/2}\rho])\r \pi'
$$
reduces (into a sum of two nonequivalent irreducible tempered representations).

Suppose that $\theta\o\s$ is a subquotient of the Jacquet module of $\pi$ such that
$$
\chi_{\theta}=\chi_{\delta([\nu^{(a_--1)/2+1}\rho, \nu^{(a-1)/2}\rho])}.
$$
Then in the proof of the above proposition, we have shown that 
$$
\theta\cong   \delta([ \nu^{(a_--1)/2+1} \rho, 
\nu^{(a-1)/2}\rho]).
$$
From $\pi\h \delta([ \nu^{-(a_--1)/2} \rho, 
\nu^{(a-1)/2}\rho])\r \pi'$ follows 
$$
\pi\h \delta([ \nu^{(a_--1)/2+1} \rho, 
\nu^{(a-1)/2}\rho])
\t
\delta([ \nu^{-(a_--1)/2} \rho, 
\nu^{(a_--1)/2}\rho])\r \pi'
$$
and further
$$
\theta\o\s \leq
M^*( \delta([ \nu^{(a_--1)/2+1} \rho, 
\nu^{(a-1)/2}\rho]))
\t
M^*(\delta([ \nu^{-(a_--1)/2} \rho, 
\nu^{(a_--1)/2}\rho]))\r \mu^*(\pi').
$$
To get $\theta\o\s$ from the right hand side of the above inequality, one needs to take from $M^*( \delta([ \nu^{(a_--1)/2+1} \rho, 
\nu^{(a-1)/2}\rho]))
$ the term $\theta\o1$ (use the formula for $Jord(\pi')$). Therefore, we need to take from 
$$
M^*(\delta([ \nu^{-(a_--1)/2} \rho, 
\nu^{(a_--1)/2}\rho]))\r\mu^*(\pi')
$$
the term $1\o\s$. It can come only from 
$$
(1\o \delta([ \nu^{-(a_--1)/2} \rho, 
\nu^{(a_--1)/2}\rho]))\r(1\o\pi')
=
1\o \delta([ \nu^{-(a_--1)/2} \rho, 
\nu^{(a_--1)/2}\rho])\r\pi'.
$$
Therefore, $\s$ is equivalent to the precisely one irreducible subrepresentation of  
$$
\delta([ \nu^{-(a_--1)/2} \rho, 
\nu^{(a_--1)/2}\rho])\r \pi'.
$$
 This implies that $\theta\o\s$   has the multiplicity one in the Jacquet module of $\pi$.

\end{remark}

We consider again an irreducible square integrable representation $\pi$ of a
classical group $S_q$. Let $\rho$ be an irreducible $F'/F$-selfdual cuspidal representation of a general linear group such that $\jrp\cap 2\mathbb Z\ne \emptyset$. Then $\ep((\rho,a))$ is defined for $a\in\jrp$. Denote
$$
a=\min(\jrp).
$$
 Assume
$$
\e_\pi((\rho,a))=1.
$$
 This is equivalent to the fact that
$$
\pi \h \delta([\nu^{1/2}\rho, \nu^{(a-1)/2}\rho])\r\sigma
$$
for some irreducible representation $\sigma$. This implies that $\delta([\nu^{1/2}\rho, \nu^{(a-1)/2}\rho])\o\sigma$ is a quotient of the Jacquet module of $\pi$. In particular,
$$
\delta([\nu^{1/2}\rho, \nu^{(a-1)/2}\rho])\o\sigma\leq \mu^*(\pi).
$$
The following proposition tells that the converse holds:

\begin{proposition}\label{Pr-eps-ev} Let $\pi$ be an irreducible square integrable representation of a
classical group $S_q$. Suppose $\jrp\cap 2\mathbb Z\ne \emptyset$. Denote
$$
a=\min(\jrp).
$$
Suppose that
$$
\delta([\nu^{1/2}\rho, \nu^{(a-1)/2}\rho])\o\sigma\leq \mu^*(\pi)
$$
for some irreducible representation $\s$.  Then
$$
\e_\pi((\rho,a))=1.
$$

\end{proposition}

\begin{proof} 
The condition in the proposition on the Jacquet module and the transitivity of Jacquet modules imply that
$$
\nu^{(a-1)/2}\rho\o\nu^{(a-1)/2-1}\rho\o\dots\o\nu^{1/2}\rho\o
\s
$$
is a subquotient in the Jacquet module of $\pi$. Now (3) of Corollary \ref{Cor-bc} implies
$$
\pi\h \nu^{(a-1)/2}\rho\t\nu^{(a-1)/2-1}\rho\t\dots\t\nu^{1/2}\rho\r
\sigma'
$$
for some irreducible representation  $\sigma'$.  Now (1) of Corollary \ref{Cor-bc} and Remark \ref{GL-inf} imply that
$$
\pi\h \theta\r\sigma''
$$
for some irreducible subquotient $\theta$ of $\nu^{(a-1)/2}\rho\t\nu^{(a-1)/2-1}\rho\t\dots\t\nu^{1/2}\rho$ and some irreducible  representation $\s''$.
Denote
$$
k_i=(a-1)/2-i+1, \quad i=1,\dots , a/2.
$$
 Then 
\begin{equation}\label{Eq-1/2}
\theta\h \nu^{k_{\a(1)}}\rho\t\dots\times\nu^{k_{\alpha({a/2})}}\rho
\end{equation}
for some permutation $\a$ of $\{1,\dots,a/2\}$.
This and $\pi\h \theta\r\s''$ imply
$$
\pi\h \nu^{k_{\a(1)}}\t\dots\times\nu^{k_{\alpha(a/2)}}\r\sigma''.
$$
Let
$$
j=\max\{1\leq i\leq a/2+1;  \alpha(s)=s \text{ for }   s=1,\dots ,i-1\}.
$$
If $j=a/2+1$, then\eqref{Eq-1/2} implies $\theta= \delta([\nu^{1/2}\rho, \nu^{(a-1)/2}\rho])$ and now $\pi\h \theta\r\s''$ implies
$$
\epsilon_{\pi}((\rho,a))=1.
$$
Suppose 
$$
j\leq a/2.
$$
 Then $j<\a(j)$ and  in \eqref{Eq-1/2} we can bring $\nu^{k_{\a(j)}}\rho$ at the "beginning" (as in the previous proof). By (1) of Proposition \ref{Pr-jb-jac} we know 
 $$
 2k_{\a(j)}+1\in\jrp.
 $$
  Note
 $$
 2k_{\a(j)}+1=2((a-1)/2-{\a(j)}+1)+1=a-2\a(j)+2.
 $$
 From $-a/2\leq -\a(j)<-j$ we get 
 $$
2=a-a+2 \leq a-2\a(j)+2 (=2k_{\a(j)}+1)<a-2j+2\leq a.
 $$
 This contradicts to the minimality of $a$ in $\jrp$.
\end{proof}

Let $\sigma$ be an irreducible cuspidal representation of a classical group and let $\rho$ be an irreducible $F'/F$-selfdual representation of a general linear group (over F'). Suppose that $\rho\r\s$ reduces. Write 
\begin{equation}
\label{pm}
\rho\r \s=\tau_1\oplus\tau_{-1}.
\end{equation}
 Then for any $k\in \mathbb Z_{>0}$ the representation $\delta([\nu\rho,\nu^k\rho])\r\tau_i$, $i\in\{\pm1\}$, has the unique irreducible subrepresentation. Denoted it by
$$
\delta([\nu\rho,\nu^k\rho]_{\tau_i};\sigma).
$$
For $k=0 $ we take formally $\delta(\emptyset_{\tau_i};\sigma)=\sigma$.

Take  an irreducible square integrable representation $\pi$ of a
classical group $S_q$ such that
$$
\pc=\s.
$$
Then $Jord_\rho(\pc)=\emptyset$.
Suppose $\jrp\ne \emptyset$. Then $\ep((\rho,a))$ is defined for $a\in \jrp$.
 Denote
$$
a=\max(\jrp).
$$
Suppose that
$$
\ep((\rho,a))=i.
$$
By \cite{T-inv}, this is equivalent to the fact that
$$
\pi\h\theta\r \delta([\nu\rho, \nu^{(a-1)/2}\rho]_{\tau_i};\pc)
$$
for some irreducible representation $\theta$ (of a general linear group). This implies that $\theta\o\delta([\nu\rho, \nu^{(a-1)/2}\rho]_{\tau_i};\pc)$ is a quotient of the Jacquet module of $\pi$. In particular.
$$
\theta\o\delta([\nu\rho, \nu^{(a-1)/2}\rho]_{\tau_i};\pc)\leq \mu^*(\pi).
$$
The following proposition implies that the converse holds:

\begin{proposition}\label{Pr-eps-0} Let $\pi$ be an irreducible square integrable representation of a
classical group $S_q$. Suppose $\jrp\ne \emptyset$ and $Jord_\rho(\pc)=\emptyset$. Denote
$$
a=\max(\jrp).
$$
Chose a decomposition into a sum of irreducible (tempered) subrepresentations:
$$
\rho\r \pc=\tau_1\oplus\tau_{-1}.
$$
Suppose that
$$
\theta\o\delta([\nu\rho, \nu^{(a-1)/2}\rho]_{\tau_i};\pc)
$$
in a subquotient of the corresponding Jacquet module of $\pi$
for some irreducible representation $\theta$.  Then
$$
\e_\pi((\rho,a))=i.
$$

\end{proposition}

\begin{proof} 
Suppose that $ \theta\o\delta([\nu\rho, \nu^{(a-1)/2}\rho]_{\tau_i};\pc)$ is is a subquotient of a Jacquet module of $\pi$ and
suppose $\e_\pi((\rho,a))=-i$. At the end of proof of Proposition 4.1 of
\cite{T-inv} we have proved that then 
$$
\theta\o\delta([\nu\rho,
\nu^{(a-1)/2}\rho]_{\tau_{-(-i)}};\pc)=\theta\o\delta([\nu\rho,
\nu^{(a-1)/2}\rho]_{\tau_{i}};\pc)
$$
is not a subquotient of the Jacquet module of $\pi$. This contradiction
completes the proof of the lemma.
\end{proof}

\begin{remark} \label{Rm-ex} (1) Let $\rho$ be an $F'/F$-selfdual irreducible cuspidal representation of a general linear group and let $\s$ be an  irreducible cuspidal representation of a classical group such that both representations
$$
\rho\r\s \text{ and }\nu\rho\r\s
$$
are irreducible. Denote by $L((\rho,\nu\rho)$)  the unique irreducible quotient of $\nu\rho\t\rho$. Then from (6.1) of \cite{T-RPI} we  directly see
\begin{equation}\label{E-1}
L((\rho,\nu\rho))\o\s\leq \mu^*(\d([\rho,\nu\rho])\r\s)
\end{equation}
and
\begin{equation}\label{E-2}
\widetilde{L((\rho,\nu\rho))}\o\s\not\leq \mu^*(\d([\rho,\nu\rho])\r\s).
\end{equation}
From Proposition 6.3 of \cite{T-RPI} we know that the representation $\d([\rho,\nu\rho])\r\s$ is irreducible.

Suppose that we have an embedding 
\begin{equation}\label{E-3}
\d([\rho,\nu\rho])\r\s\h 
L((\rho,\nu\rho))\r\s.
\end{equation}
Then Corollary 6.4 of \cite{T-RPI} implies that we have also embedding
\begin{equation}\label{E-4}
\d([\rho,\nu\rho])\r\s\h 
\widetilde{L((\rho,\nu\rho))}\r\s.
\end{equation}
Now Frobenius reciprocity implies that $\widetilde{L((\rho,\nu\rho))}\o\s$ is a quotient of the corresponding Jacquet module of $\d([\rho,\nu\rho])\r\s$. In particular $\widetilde{L((\rho,\nu\rho))}\o\s\leq \mu^*(\d([\rho,\nu\rho])\r\s$). This contradicts \eqref{E-2}. Therefore, \eqref{E-3} can not hold.

The simplest example of the above situation is when we take $\rho=\mathbf 1_{F^\times}$ and $\sigma=\mathbf 1_{SO(1,F)}$ ($\mathbf 1_G$ denotes the trivial one-dimensional representation, while $St_G$ denotes the Steinberg representation of a reductive $p$-adic group $G$). Then $\nu^{1/2}St_{GL(2,F)}\r \mathbf 1_{SO(1,F)}$ is an irreducible representation of $SO(5,F)$. This representation has $\nu^{1/2}\mathbf 1_{GL(2,F)}\o \mathbf 1_{SO(1,F)}$ for a subquotient in the corresponding Jacquet module, but it does not embed into  $\nu^{1/2}\mathbf 1_{GL(2,F)}\r \mathbf 1_{SO(1,F)}$.

\bigskip

\noindent(2) Let $\rho$ be an irreducible
unitarizable cuspidal representation of  $GL(p,F)$ and let 
$\sigma$ be  an irreducible cuspidal representation of 
a classical group. Suppose that $\beta>1/2$ is in $(1/2)\,\mathbb Z$ and that  $\nu^\beta\rho
\rtimes \sigma$ reduces.
Then Proposition 5.1 of \cite{T-RPI} says that representations
$
\nu^{\beta} \rho  \rtimes \delta(\nu^\beta\rho,\sigma)
$
and
$
\nu^\beta\rho \rtimes L(\nu^\beta\rho,\sigma)
$
are irreducible. Compute
$$
\mu^*(\nu^{\beta} \rho  \rtimes \delta(\nu^\beta\rho,\sigma))=
(1\o\nu^{\beta} \rho +\nu^{\beta} \rho \o1+\nu\rho^{-\beta}\o1)\r(1\o \delta(\nu^\beta\rho,\sigma) + \nu^\beta\rho\o\sigma)
$$
$$
=1\o \nu^{\beta} \rho  \rtimes \delta(\nu^\beta\rho,\sigma)+[ \nu^{\beta} \rho  \o \delta(\nu^\beta\rho,\sigma)+\nu^{\beta} \rho  \o \nu^\beta\rho\r\sigma+\nu^{-\beta}\rho \o \delta(\nu^\beta\rho,\sigma)  ] 
$$
$$
+ [\nu^{\beta} \rho \t \nu^{\beta} \rho \o\s +\nu^{\beta} \rho \t\nu^{-\beta} \rho \o\s  ].
$$
Observe that $\nu^{\beta} \rho \o L(\nu^{\beta} \rho ;\s)\leq \mu^*(\nu^{\beta} \rho  \rtimes \delta(\nu^\beta\rho,\sigma))$ but   $\nu^{-\beta} \rho \o L(\nu^{\beta} \rho ;\s)\not\leq \mu^*(\nu^{\beta} \rho  \rtimes \delta(\nu^\beta\rho,\sigma))$. 

Suppose 
$$
 \nu^{\beta} \rho  \rtimes \delta(\nu^\beta\rho,\sigma)\h 
 \nu^{\beta} \rho \r L(\nu^{\beta} \rho ;\s). 
 $$
 Then $ \nu^{\beta} \rho  \rtimes \delta(\nu^\beta\rho,\sigma)\h 
 \nu^{-\beta} \rho \r L(\nu^{\beta} \rho ;\s)$, which implies that $\nu^{-\beta} \rho \o L(\nu^{\beta} \rho ;\s)$ is a subquotient of the corresponding Jacquet module of $\nu^{\beta} \rho  \rtimes \delta(\nu^\beta\rho,\sigma)$. This contradicts  $\nu^{-\beta} \rho \o L(\nu^{\beta} \rho ;\s)\not\leq \mu^*(\nu^{\beta} \rho  \rtimes \delta(\nu^\beta\rho,\sigma))$. Therefore $ \nu^{\beta} \rho  \rtimes \delta(\nu^\beta\rho,\sigma)$ does not embed into  
$ \nu^{\beta} \rho \r L(\nu^{\beta} \rho ;\s)$.

The simplest example of the above situation is when we take $\rho=\mathbf 1_{F^\times}$ and $\sigma=\mathbf 1_{Sp(0,F)}$. Then $\nu\mathbf 1_{F^\times}\r St_{Sp(2,F)}$ is an irreducible representation of $Sp(4,F)$. This representation has $\nu\mathbf 1_{F^\times}\o \mathbf 1_{Sp(2,F)}$ for a subquotient in its corresponding Jacquet module, but it does not embed into  $\nu\mathbf 1_{F^\times}\t \mathbf 1_{Sp(2,F)}$ .

\end{remark}

\section{Behavior of partially defined function for deforming  Jordan blocks}\label{behaviour}
\setcounter{equation}{0}

In Proposition \ref{Pr-jb-main},  square integrable representations of a smaller and bigger classical groups are related. There are two types of relation (see (1) and (2) of Proposition \ref{Pr-jb-main}). The proposition describes the behavior of the Jordan blocks in both cases. Note that the  partial cuspidal supports are preserved. Theorem \ref{Th-eps=} of C. M\oe glin describes the behavior of the partially defined function for the relation of type (2) of that proposition. In this section we shall describe the behavior of the partially defined function for the relation of type  (1) of that proposition.

\begin{lemma}\label{Le-red} Let $\pi$ be an irreducible square integrable representation of a
classical group $S_q$. Suppose that we have
$a\in
\jrp$, $a\geq 3$, which satisfies
$a-2\not\in
\jrp$. Then there exists an irreducible square integrable representation $\pi'$ such that
\begin{equation}\label{B1}
\pi\h \nu^{(a-1)/2}\rho\r \pi'.
\end{equation}
Let $\pi'$ be  any irreducible square integrable representation  satisfying \eqref{B1}.
Then:

 \begin{enumerate}

\item  $\pi$ is the unique irreducible subrepresentation of $ \nu^{(a-1)/2}\rho\r \pi'$.
\item
$$
\pi_{cusp}'= \pi_{cusp}.
$$
\item
$$ Jord(\pi')=\left(Jord(\pi)\ \backslash\ \{(\rho,a)\}\right)\cup
\{(\rho,a-2)\}.
$$

\item
Let 
$$
(\rho',b),(\rho',c)\in Jord(\pi)
$$
 (the possibility $b=c$ is not excluded). Observe that this implies $(\rho,a-2)\notin\{(\rho',b),(\rho',c)\}$.

Suppose $\rho'\not\cong\rho$, or $\rho'\cong\rho$ but $a\notin\{b,c\}$.  If $b\ne c$, then
\begin{equation}\label{Eq-B2}
\epsilon_{\pi'}((\rho',b))\epsilon_{\pi'}((\rho',c))^{-1}=
\epsilon_{\pi}((\rho',b))\epsilon_{\pi}((\rho',c))^{-1}.
\end{equation}
Further,
$\epsilon_{\pi'}((\rho',b))$ is defined if and only if
$\epsilon_\pi((\rho',b))$ is defined. If it is defined,  then 
\begin{equation}\label{Eq-B3}
\epsilon_{\pi'}((\rho',b))=\epsilon_{\pi}((\rho',b)).
\end{equation}

Suppose $\rho'\cong\rho$. If $b\ne a$, then
\begin{equation}\label{Eq-B4}
\epsilon_{\pi'}((\rho,b))\epsilon_{\pi'}((\rho,a-2))^{-1}=
\epsilon_{\pi}((\rho,b))\epsilon_{\pi}((\rho,a))^{-1}.
\end{equation}
Further,
$\epsilon_{\pi'}((\rho,a-2))$ is defined if and only if
$\epsilon_\pi((\rho,a))$ is defined. If it is defined,  then 
\begin{equation}\label{Eq-B5}
\epsilon_{\pi'}((\rho,a-2))=\epsilon_{\pi}((\rho,a)).
\end{equation}
\item If $\sigma$ is an irreducible  representation of a classical group such that
$$
\pi\h \nu^{(a-1)/2}\rho\r \sigma,
$$ 
then $\sigma\cong\pi'$. In particular, $\sigma$  is uniquely
determined by $\pi$ (and it is square integrable).

\end{enumerate}
\end{lemma}

Note that  (4) tells us that  $\epsilon_{\pi'}$ is  completely determined by $\ep$.

\begin{proof} The existence of an irreducible square integrable representation $\pi'$ satisfying
\eqref{B1}  is just Lemma in 10.2.2 of \cite{Moe-T} (see there also Remark   after that
lemma). It also implies  that  claim (3) holds. Now Lemma 5.3 of \cite{Moe-T} applied to $\pi'$ gives (1). Claim (2) follows directly from the
definition of the cuspidal support: $\pi\h\nu^{(a-1)/2}\rho\t \theta\r\pi'_{cusp}$ for
some irreducible representation $\theta$ of a general linear group.
It remains to prove (4) and (5).

First we shall prove (5). Suppose $\pi\h \nu^{(a-1)/2}\rho\r \sigma$. This and \eqref{B1}
imply
\begin{equation}
\label{Eq-B55}
 \nu^{(a-1)/2}\rho\o \sigma\leq \mu^*(\pi) \leq \mu^*( \nu^{(a-1)/2}\rho\r \pi').
\end{equation}
We have
\begin{equation}
\label{Eq-B555}
\mu^*( \nu^{(a-1)/2}\rho\r \pi')=(1\o\nu^{(a-1)/2}\rho+ \nu^{(a-1)/2}\rho \o1
+ \nu^{-(a-1)/2}\rho\o1)
\r \mu^*(\pi').
\end{equation}
Now \eqref{Eq-B55} and the above formula imply
\begin{equation}\label{Eq-B6}
 \nu^{(a-1)/2}\rho\o \sigma
 \leq 
 (1\o\nu^{(a-1)/2}\rho)
\r \mu^*(\pi')
+
 (\nu^{(a-1)/2}\rho\o1)
\r \mu^*(\pi').
\end{equation}
Suppose that
\begin{equation}\label{Eq-B66}
 \nu^{(a-1)/2}\rho\o \sigma
 \leq 
 (\nu^{(a-1)/2}\rho\o1)
\r \mu^*(\pi').
\end{equation}
Then the formula for $\mu^*(\pi')$ implies $\sigma\cong\pi'$, so (5) holds. Suppose that
\eqref{Eq-B66} does not hold. In that case 
$$
\nu^{(a-1)/2}\rho\o \sigma\leq (1\o \nu^{(a-1)/2}\rho )
\r \mu^*(\pi').
$$
This implies that that there exits an irreducible representation $\tau\otimes\tau'\leq s_{(p)}(\sigma)$ such
that 
$$
 \nu^{(a-1)/2}\rho\o \sigma\leq\tau\o \nu^{(a-1)/2}\rho\r\tau'.
$$
This implies $\tau\cong \nu^{(a-1)/2}\rho$. Now (1) of Proposition \ref{Pr-jb-jac}  implies that
($\rho,a)\in Jord(\pi')$, which contradicts (3). This completes the proof of (5).

Now we shall prove (4). 
The proof proceeds in a number of steps.

\medskip

\noindent
{\bf (A)} Suppose 
$$
b=\min(Jord_{\rho'}(\pi')).
$$
Recall $(\rho',b)\ne (\rho,a-2)$ (since $(\rho',b)\in \jp$ and $(\rho,a-2)\notin \jp$).
Therefore
$$
b=\min(Jord_{\rho'}(\pi)).
$$
Clearly, $(\rho',b)\ne(\rho,a)$.
 Assume additionally 
 $$
 b\in2\mathbb Z.
 $$
\begin{enumerate}
\item
Let
$
\e_{\pi'}((\rho',b))=1.
$
 Then
$$
\pi'\h \d([\nu^{1/2}\rho',\nu^{(b-1)/2}\rho'])\r\sigma
$$
for some irreducible representation $\sigma.$ Now
$$
\pi\h\nu^{(a-1)/2}\rho\t \pi'\h
\nu^{(a-1)/2}\rho\t\d([\nu^{1/2}\rho',\nu^{(b-1)/2}\rho'])\r\sigma.
$$
The condition $(\rho',b)\ne(\rho,a-2)$ implies
$$
\pi\h
\d([\nu^{1/2}\rho',\nu^{(b-1)/2}\rho'])\t\nu^{(a-1)/2}\rho\r\sigma.
$$
 Thus
$
\e_{\pi}((\rho',b))=1.
$

\item
Let now
$
\e_{\pi}((\rho',b))=1.
$
 Then
$$
\pi\h \d([\nu^{1/2}\rho',\nu^{(b-1)/2}\rho'])\r\sigma
$$
for some irreducible representation $\sigma.$ This implies
$$
\d([\nu^{1/2}\rho',\nu^{(b-1)/2}\rho'])\o\sigma
\leq\mu^*(\nu^{(a-1)/2}\rho\r\pi')
\hskip40mm
$$
$$
\hskip30mm
=(1\o\nu^{(a-1)/2}\rho+ \nu^{(a-1)/2}\rho \o1
+ \nu^{-(a-1)/2}\rho\o1)
\r \mu^*(\pi').
$$
Now the above formula and $(\rho',b)\ne(\rho,a)$, $(\rho',b)\ne(\rho,a-2)$  directly imply that
$$
\d([\nu^{1/2}\rho',\nu^{(b-1)/2}\rho'])\o\sigma
\leq
(1\o\nu^{(a-1)/2}\rho)\r
\mu^*(\pi').
$$
From this follows that $\d([\nu^{1/2}\rho',\nu^{(b-1)/2}\rho'])\o\s'\leq \varphi\o\nu^{(a-1)/2}\rho\r\psi$ for some irreducible subquotient  $\varphi\o\psi\leq \mu^*(\pi')$. Clearly, $\d([\nu^{1/2}\rho',\nu^{(b-1)/2}\rho'])\cong \varphi$.
Now  from Proposition \ref{Pr-eps-ev}  we get
$
\e_{\pi'}((\rho',b))=1.
$
\end{enumerate}

\noindent
{\bf (B)} Suppose 
$$
a-2=\min(Jord_{\rho}(\pi')).
$$
Then also
$a=\min(Jord_{\rho}(\pi))$ (and conversely).  Assume additionally 
$$
a\in 2\mathbb Z.
$$

\begin{enumerate}
\item
Suppose
$
\e_{\pi'}((\rho,a-2))=1.
$
 Then
$$
\pi'\h \d([\nu^{1/2}\rho,\nu^{(a-3)/2}\rho])\r\sigma
$$
for some irreducible representation $\sigma.$ Now
$$
\pi\h\nu^{(a-1)/2}\rho\t \pi'\h
\nu^{(a-1)/2}\rho\t\d([\nu^{1/2}\rho,\nu^{(a-3)/2}\rho])\r\sigma.
$$
Now $\nu^{(a-1)/2}\rho\o\d([\nu^{1/2}\rho,\nu^{(a-3)/2}\rho])\o\sigma$ is
in the Jacquet module of $\pi$. Transitivity of Jacquet modules implies that
 also $\d([\nu^{1/2}\rho,\nu^{(a-1)/2}\rho])\o\sigma$ must be a subquotient of the Jacquet module of $\pi$. Now Proposition \ref{Pr-eps-ev} 
implies
$
\e_{\pi}((\rho,a))=1.
$

\item
Let now
$
\e_{\pi}((\rho,a))=1.
$
 Then
$$
\pi\h \d([\nu^{1/2}\rho,\nu^{(a-1)/2}\rho])\r\sigma
$$
for some irreducible representation $\sigma.$ This implies
$$
\d([\nu^{1/2}\rho,\nu^{(a-1)/2}\rho])\o\sigma
\leq\mu^*(\nu^{(a-1)/2}\rho\r\pi')
\hskip40mm
$$
$$
\hskip30mm
=(1\o\nu^{(a-1)/2}\rho+ \nu^{(a-1)/2}\rho \o1
+ \nu^{-(a-1)/2}\rho\o1)
\r \mu^*(\pi').
$$
From this  follows that 
\begin{equation}
\label{Eq-add-2}
\d([\nu^{1/2}\rho,\nu^{(a-3)/2}\rho])\o\sigma'
\leq\mu^*(\pi')
\end{equation}
or 
\begin{equation*}
\d([\nu^{1/2}\rho,\nu^{(a-1)/2}\rho])\o\sigma''
\leq\mu^*(\pi')
\end{equation*}
for some irreducible representations $\sigma'$ and $\sigma''$. The last inequality implies $(\rho,a)\in Jord(\pi')$, which is impossible. Therefore, \eqref{Eq-add-2} holds.
Now Proposition \ref{Pr-eps-ev} implies
$
\e_{\pi'}((\rho,a-2))=1.
$
\end{enumerate}

\noindent
{\bf (C)}
Suppose  
$$
\text{$\rho'\not\cong\rho$, or $\rho'\cong\rho$
but
$a\notin\{b,c\}$.  }
$$
Let $b\ne c$. Consider the case 
$
b=c_-.
$

\begin{enumerate}
\item
First suppose that
$
\epsilon_{\pi'}((\rho',c_-))\epsilon_{\pi'}((\rho',c))^{-1}=1.
$
Then by Remark 5.1.3 of \cite{Moe-Ex} (or Lemma 5.1 of \cite{Moe-T}) we have
$$
\pi'\h \delta([\nu^{-(c_--1)/2}\rho', \nu^{(c-1)/2}\rho'])\r\pi''
$$
for some irreducible square integrable representation $\pi''$.
Now
$$
\pi \h \nu^{(a-1)/2}\rho\r\pi'\h \nu^{(a-1)/2}\rho\t\delta([\nu^{-(c_--1)/2}\rho',
\nu^{(c-1)/2}\rho'])\r\pi''.
$$
Suppose that the segments $\{\nu^{(a-1)/2}\rho\}$ and
$\delta([\nu^{-(c_--1)/2}\rho',
\nu^{(c-1)/2}\rho'])$ are  linked. In this case $\rho'\cong\rho$ and we have two possibilities.
Then first is $(c-1)/2+1=(a-1)/2$, i.e. $c=a-2$, which is not possible since $c\in Jord_\rho(\pi)$
and
$a-2\notin Jord_\rho(\pi)$. The second possibility is ($a-1)/2+1=-(c_--1)/2$. This is obviously
impossible.
Therefore, the segments $\{\nu^{(a-1)/2}\rho\}$ and $\delta([\nu^{-(c_--1)/2}\rho',
\nu^{(c-1)/2}\rho'])$ are not linked. This implies
$$
\hskip12mm
\nu^{(a-1)/2}\rho\t\delta([\nu^{-(c_--1)/2}\rho',
\nu^{(c-1)/2}\rho'])
\cong
\delta([\nu^{-(c_--1)/2}\rho',
\nu^{(c-1)/2}\rho'])\t \nu^{(a-1)/2}\rho.
$$
Therefore
$$
\pi
\h
\delta([\nu^{-(c_--1)/2}\rho',
\nu^{(c-1)/2}\rho'])\t \nu^{(a-1)/2}\rho\r\pi''
\h
\hskip
35mm
$$
$$
\hskip27mm
\delta([\nu^{(c_-+1)/2}\rho',
\nu^{(c-1)/2}\rho'])
\t
\delta([\nu^{-(c_-1)/2}\rho',
\nu^{(c_--1)/2}\rho'])\t \nu^{(a-1)/2}\rho\r\pi''.
$$
Now  the  definition of the partially defined function implies
$
\epsilon_{\pi}((\rho',c_-))\epsilon_{\pi}((\rho,c))^{-1}$ $=1.
$
Therefore in this case we have
$
\epsilon_{\pi'}((\rho',c_-))\epsilon_{\pi'}((\rho',c))^{-1}=
\epsilon_{\pi}((\rho',c_-))\epsilon_{\pi}((\rho',c))^{-1}.
$

\item
Suppose  now 
$
\epsilon_{\pi}((\rho',c_-))\epsilon_{\pi}((\rho',c))^{-1}=1.
$
Then
\begin{equation}\label{Eq-B7}
\pi\h \delta([\nu^{-(c_--1)/2}\rho', \nu^{(c-1)/2}\rho'])\r\pi''
\end{equation}
for some irreducible square integrable representation $\pi''$. 
Now \eqref{Eq-B7} and
$
\pi \h \nu^{(a-1)/2}\rho\r\pi'
$
imply
$$
\delta([\nu^{-(c_--1)/2}\rho', \nu^{(c-1)/2}\rho'])\o\pi''
\leq
\mu^*(\nu^{(a-1)/2}\rho\r\pi') \hskip40mm
$$
$$
\hskip30mm= (1\o\nu^{(a-1)/2}\rho+ \nu^{(a-1)/2}\rho \o1
+\nu^{-(a-1)/2}\rho\o1)
\r \mu^*(\pi').
$$
Suppose
$$
\delta([\nu^{-(c_--1)/2}\rho', \nu^{(c-1)/2}\rho'])\o\pi''
\leq
 (\nu^{\pm(a-1)/2}\rho\o1)
\r \mu^*(\pi').
$$
Then $\rho'\cong\rho$ and $a\leq c$, which implies $a<c$ since we consider the case $a\ne c$.
Therefore, $a<c_-$ also (since $a\ne b$). In this case we  have two possibilities (corresponding to the choice of sign in the term $\nu^{\pm(a-1)/2}\rho\o1$). The first possibility implies
$$
\delta([\nu^{-(c_--1)/2}\rho', \nu^{(a-1)/2-1}\rho']) \t \delta([\nu^{(a-1)/2+1}\rho',
\nu^{(c-1)/2}\rho'])\o \sigma\leq \mu^*(\pi')
$$
for some irreducible representation $\sigma$. This would imply that $\pi'$ is not square integrable.
The second possibility implies in the same way that $\pi'$ is not square integrable.

Thus
$$
\delta([\nu^{-(c_--1)/2}\rho', \nu^{(c-1)/2}\rho'])\o\pi''
\leq
(1\o \nu^{(a-1)/2}\rho 
)
\r \mu^*(\pi').
$$
This implies  that
$$
\delta([\nu^{-(c_--1)/2}\rho', \nu^{(c-1)/2}\rho'])\o\s'\leq \mu^*(\pi')
$$
for some irreducible representation $\s'$. Now Proposition \ref{Pr-jb-jac} easily implies that
$
\epsilon_{\pi'}((\rho',c_-))\epsilon_{\pi'}((\rho,c))^{-1}=1.
$
Thus 
 $
\epsilon_{\pi'}((\rho',c_-))\epsilon_{\pi'}((\rho',c))^{-1}=
\epsilon_{\pi}((\rho',c_-))\epsilon_{\pi}((\rho',c))^{-1}
$
holds also in this case.

\end{enumerate}

\noindent
{\bf (D)} Suppose that  $a_-$ is defined. Then $(a-2)_-$ is defined, and
conversely. In that case 
$$
(a-2)_-=a_-.
$$

\begin{enumerate}
\item
Suppose
$
\epsilon_{\pi'}((\rho,a_-))\epsilon_{\pi'}((\rho,a-2))^{-1}=1.
$
Then
$$
\pi'\h \delta([\nu^{-(a_--1)/2}\rho, \nu^{(a-3)/2}\rho])\r\pi''
$$
for some irreducible square integrable representation $\pi''$.
Now
$$
\pi \h \nu^{(a-1)/2}\rho\r\pi'\h
\nu^{(a-1)/2}\rho\t\delta([\nu^{-(a_--1)/2}\rho,
\nu^{(a-3)/2}\rho])\r\pi''.
$$
In a standard way we get that $\delta([\nu^{(a_--1)/2+1}\rho,
\nu^{(a-1)/2}\rho])\o\pi''$ is a subquotient of the Jacquet module of $\pi$, and now (1) of Proposition \ref{Pr-eps=}
implies
$
\epsilon_{\pi}((\rho,a_-))\epsilon_{\pi}((\rho,a))^{-1}=1.
$
Therefore in this case holds
$
\epsilon_{\pi'}((\rho,a_-))\epsilon_{\pi'}((\rho,a-2))^{-1}=
\epsilon_{\pi}((\rho,a_-))\epsilon_{\pi}((\rho,a))^{-1}.
$

\item
 Suppose  now 
$
\epsilon_{\pi}((\rho,a_-))\epsilon_{\pi}((\rho,a))^{-1}=1.
$
Then
\begin{equation}\label{Eq-B77}
\pi\h \delta([\nu^{-(a_--1)/2}\rho, \nu^{(a-1)/2}\rho])\r\pi''
\end{equation}
for some irreducible square integrable representation $\pi''$. 
Now \eqref{Eq-B77} and
$
\pi \h \nu^{(a-1)/2}\rho\r\pi'
$
imply
$$
\delta([\nu^{-(a_--1)/2}\rho, \nu^{(a-1)/2}\rho])\o\pi''
\leq
\mu^*(\nu^{(a-1)/2}\rho\r\pi') \hskip40mm
$$
$$
\hskip30mm= (1\o\nu^{(a-1)/2}\rho+ \nu^{(a-1)/2}\rho \o1
+\nu^{-(a-1)/2}\rho\o1)
\r \mu^*(\pi').
$$
Suppose
$$
\delta([\nu^{-(a_--1)/2}\rho, \nu^{(a-1)/2}\rho])\o\pi''
\leq
 (\nu^{\pm(a-1)/2}\rho\o1)
\r \mu^*(\pi').
$$
 In this case we must have
$$
\delta([\nu^{-(a_--1)/2}\rho, \nu^{(a-1)/2-1}\rho]) 
\o \sigma\leq \mu^*(\pi')
$$
for some irreducible representation $\sigma$.  Now  Proposition \ref{Pr-eps=} implies that we have
$\epsilon_{\pi'}((\rho,a_-))\epsilon_{\pi'}((\rho,a-2))^{-1}=1$. It remains to consider the case
$$
\delta([\nu^{-(a_--1)/2}\rho, \nu^{(a-1)/2}\rho])\o\pi''
\leq
(1\o \nu^{(a-1)/2}\rho 
)
\r \mu^*(\pi').
$$
Now one directly gets that
$$
\delta([\nu^{-(a_--1)/2}\rho, \nu^{(a-1)/2}\rho])\o\s'\leq \mu^*(\pi')
$$
for some irreducible representation $\s'$. From (1) of Proposition \ref{Pr-jb-jac} now follows
$a\in Jord_\rho(\pi')$, which is a contradiction.

Therefore, we have proved that also in this case we have 
$
\epsilon_{\pi'}((\rho,a-2))\epsilon_{\pi'}((\rho,a_-))^{-1}$
$
=
\epsilon_{\pi}((\rho,a))\epsilon_{\pi}((\rho,a_-))^{-1}.
$
\end{enumerate}

\noindent
{\bf (E)} Suppose that 
$$
\text{$\rho'\cong\rho$ and $b_-=a$ is defined in $\jrp$.}
$$
Then
$b_-$ is $a-2$ in $Jord_\rho(\pi')$ (the converse also holds). We denote 
$$
a_+=b.
$$

\begin{enumerate}
 \item
Suppose
$
\epsilon_{\pi'}((\rho,a_+))\epsilon_{\pi'}((\rho,a-2))^{-1}=1.
$
Then
$$
\pi'\h \delta([\nu^{-(a-3)/2}\rho, \nu^{(a_+-1)/2}\rho])\r\pi''
$$
for some irreducible square integrable representation $\pi''$. 
Now
$$
\pi \h \nu^{(a-1)/2}\rho\r\pi'\h
\nu^{(a-1)/2}\rho\t\delta([\nu^{-(a-3)/2}\rho,
\nu^{(a_+-1)/2}\rho])\r\pi''.
$$
In standard way we get that $\delta([\nu^{(a-1)/2+1}\rho,
\nu^{(a_+-1)/2}\rho])\o\pi''$ is a subquotient of the Jacquet module of $\pi$. Now Proposition \ref{Pr-eps=}
implies
$
\epsilon_{\pi}((\rho,a_+))\epsilon_{\pi}((\rho,a))^{-1}=1.
$
Therefore  we have 
$
\epsilon_{\pi'}((\rho,a_+))\epsilon_{\pi'}((\rho,a-2))^{-1}=
\epsilon_{\pi}((\rho,a_+))\epsilon_{\pi}((\rho,a))^{-1}.
$

\item
Suppose  now 
$
\epsilon_{\pi}((\rho,a_+))\epsilon_{\pi}((\rho,a))^{-1}=1.
$
Then
\begin{equation}\label{Eq-B77+}
\pi\h \delta([\nu^{-(a-1)/2}\rho, \nu^{(a_+-1)/2}\rho])\r\pi''
\end{equation}
for some irreducible square integrable representation $\pi''$. 
Now \eqref{Eq-B77+} and
$
\pi \h \nu^{(a-1)/2}\rho\r\pi'
$
imply
$$
\delta([\nu^{-(a-1)/2}\rho, \nu^{(a_+-1)/2}\rho])\o\pi''
\leq
\mu^*(\nu^{(a-1)/2}\rho\r\pi') \hskip40mm
$$
$$
\hskip30mm= (1\o\nu^{(a-1)/2}\rho+ \nu^{(a-1)/2}\rho \o1
+\nu^{-(a-1)/2}\rho\o1)
\r \mu^*(\pi').
$$
Suppose
$$
\delta([\nu^{-(a-1)/2}\rho, \nu^{(a_+-1)/2}\rho])\o\pi''
\leq
 (\nu^{\pm(a-1)/2}\rho\o1)
\r \mu^*(\pi').
$$
 In this case we must have
$$
\delta([\nu^{-(a-1)/2}\rho, \nu^{(a-1)/2-1}\rho]) \t
\delta([\nu^{(a-1)/2+1}\rho,
\nu^{(a_+-1)/2}\rho])\o \sigma\leq \mu^*(\pi')
$$
or
$$
\delta([\nu^{-(a-1)/2+1}\rho, \nu^{(a_+-1)/2}\rho])
\o \sigma'
\leq \mu^*(\pi')
$$
for some irreducible representations $\sigma$. The first possibility would imply that $\pi'$ is not square
integrable, which is contradiction. The second possibility and Proposition \ref{Pr-eps=} easily imply that
$
\epsilon_{\pi'}((\rho,a-2))\epsilon_{\pi'}((\rho,a_+))^{-1}=1.
$

It remains to consider the case 
$$
\delta([\nu^{-(a-1)/2}\rho, \nu^{(a_+-1)/2}\rho])\o\pi''
\leq
(1\o \nu^{(a-1)/2}\rho 
)
\r \mu^*(\pi').
$$
Now one directly gets that
$$
\delta([\nu^{(a-1)/2+1}\rho, \nu^{(a_+-1)/2}\rho])\o\s''\leq\mu^*(\pi')
$$
for some irreducible representation $\s''$. Proposition \ref{Pr-eps=} now implies
$
\epsilon_{\pi'}((\rho,a-2))\epsilon_{\pi'}((\rho,a_+))^{-1}=1.
$
Therefore in this case  also holds
$
\epsilon_{\pi'}((\rho,a-2))\epsilon_{\pi'}((\rho,a_+))^{-1}=
\epsilon_{\pi}((\rho,a))\epsilon_{\pi}((\rho,a_+))^{-1}.
$
\end{enumerate}

\noindent
{\bf (F)} Suppose 
$$
b=\max(Jord_{\rho'}(\pi')).
$$
 Then
 $(\rho',b)\ne(\rho,a)$ since $(\rho,a)\notin Jord(\pi')$, and therefore
$b=\max(Jord_{\rho'}(\pi))$ (and conversely). 
Suppose additionally
$$
Jord_{\rho'}(\pc)=\emptyset.
$$
Write
\begin{equation}
\label{Eq-B777}
\rho'\r\pc=\tau_1'\oplus\tau_{-1}'.
\end{equation}
Let
$$
\e_{\pi'}((\rho',b))=i.
$$
 Then
$$
\pi'\h \theta \r
\d([\nu\rho',\nu^{(b-1)/2}\rho']_{\tau_i'};\pc)
$$
for some irreducible representation $\theta.$ Now
$$
\pi\h\nu^{(a-1)/2}\rho\r \pi'\h
\nu^{(a-1)/2}\rho\t\theta \r
\d([\nu\rho',\nu^{(b-1)/2}\rho']_{\tau_i'};\pc).
$$
Then $\e_{\pi}((\rho',b))=i$, and
therefore
$
\e_{\pi'}((\rho',b))=\e_{\pi}((\rho',b)).
$

\noindent
{\bf (G)} Suppose 
$$
a-2=\max(Jord_{\rho}(\pi')).
$$
 Then 
$a=\max(Jord_{\rho}(\pi))$ (and conversely). Suppose
$$
Jord_\rho(\pc)=\emptyset
$$
 and assume the decomposition    $\rho'\r\pc=\tau_1'\oplus\tau_{-1}'$ from \eqref{Eq-B777} to hold.

Let 
$$
\e_{\pi}((\rho,a))=i.
$$
 Then
$$
\pi\h \theta\r
\d([\nu\rho,\nu^{(a-1)/2}\rho]_{\tau_i};\pc)
$$
for some irreducible representation $\theta.$
 This implies
\begin{equation}
\label{Eq-B7777} 
\theta\o\d([\nu\rho,\nu^{(a-1)/2}\rho]_{\tau_i},\pc)
\leq\mu^*(\pi')
\end{equation}
for some (irreducible) representation  $\theta$. This implies
$$
\theta\o\d([\nu\rho,\nu^{(a-1)/2}\rho]_{\tau_i};\pc)
\leq\mu^*(\nu^{(a-1)/2}\rho\r\pi')
\hskip40mm
$$
$$
\hskip30mm
=(1\o\nu^{(a-1)/2}\rho+ \nu^{(a-1)/2}\rho \o1
+ \nu^{-(a-1)/2}\rho\o1)
\r \mu^*(\pi').
$$
From (ii) of Proposition 5.2 from \cite{T-seg} we know that
$
\nu^{(a-1)/2}\rho
\o
\d([\nu\rho,\nu^{(a-3)/2}\rho]_{\tau_i};\pc)
\leq
\mu^*(\d([\nu\rho,\nu^{(a-1)/2}\rho]_{\tau_i};\pc)).
$
From this and \eqref{Eq-B7777} follow that 
$$
\theta'\o\d([\nu\rho,\nu^{(a-3)/2}\rho]_{\tau_i};\pc)
\leq\mu^*(\pi')
$$
for some irreducible representation $\theta'$. 
Now Proposition \ref{Pr-eps-ev} implies
$
\e_{\pi'}((\rho,a-2))=i,
$
which  gives
$
\e_{\pi'}((\rho,a-2))=\e_{\pi}((\rho,a)).
$

The proof of lemma is now complete.
\end{proof}

\begin{theorem}\label{Th-red} Let $\pi$ be an irreducible square integrable representation of a
classical group $S_q$, let $\rho$ be an irreducible cuspidal $F'/F$-selfdual representation of $GL(p,F')$ and let 
$a\in
\jrp$, $a\geq 3$. Suppose that there exists $k\in \mathbb Z_{>0}$ such that
\begin{equation}\label{Eq-C1--}
[a-2k,a-2]\cap \jrp=\emptyset.
\end{equation}
 Then there exists an irreducible square integrable representation $\pi'$ of $S_{q-kp}$ such that $\pi$ embeds into
\begin{equation}\label{Eq-C1-}
\d([
\nu^{(a-2(k-1)-1)/2}\rho,
 \nu^{(a-1)/2}\rho])
\r \pi'.
\end{equation}
Then $\pi$ embeds also into
\begin{equation}\label{Eq-C1}
 \nu^{(a-1)/2}\rho\t
\nu^{(a-3)/2}\rho\t
\dots
\t
\nu^{(a-2(k-1)-1)/2}\rho
\r \pi'.
\end{equation}
Let $\pi'$ be  any irreducible square integrable representation  such that $\pi$ embeds into  \eqref{Eq-C1}.
Then:

 \begin{enumerate}

\item  $\pi$ is the unique irreducible subrepresentation of 
$$
\nu^{(a-1)/2}\rho\t
\nu^{(a-3)/2}\rho\dots
\nu^{(a-2(k-1)-1)/2}\rho
\r \pi'.
$$
\item
$$
\pi_{cusp}'= \pi_{cusp}.
$$
\item
\begin{equation}
\label{Eq-C11}
 Jord(\pi')=\left(Jord(\pi)\ \backslash\ \{(\rho,a)\}\right)\cup
\{(\rho,a-2k)\}.
\end{equation}

\item
Let $(\rho',b),(\rho',c)\in Jord(\pi)$ (the possibility $b=c$ is not excluded).

Suppose $\rho'\not\cong\rho$, or $\rho'\cong\rho$ but $a\notin\{b,c\}$.  If $b\ne c$, then
\begin{equation}\label{Eq-C2}
\epsilon_{\pi'}((\rho',b))\epsilon_{\pi'}((\rho',c))^{-1}=
\epsilon_{\pi}((\rho',b))\epsilon_{\pi}((\rho',c))^{-1}.
\end{equation}
Further,
$\epsilon_{\pi'}((\rho',b))$ is defined if and only if
$\epsilon_\pi((\rho',b))$ is defined. If it is defined,  then 
\begin{equation}\label{Eq-C3}
\epsilon_{\pi'}((\rho',b))=\epsilon_{\pi}((\rho',b)).
\end{equation}

Suppose $\rho'\cong\rho$. If $b\ne a$, then
\begin{equation}\label{Eq-C4}
\epsilon_{\pi'}((\rho,b))\epsilon_{\pi'}((\rho,a-2k))^{-1}=
\epsilon_{\pi}((\rho,b))\epsilon_{\pi}((\rho,a))^{-1}.
\end{equation}
Further,
$\epsilon_{\pi'}((\rho,a-2k))$ is defined if and only if
$\epsilon_\pi((\rho,a))$ is defined. If it is defined,  then 
\begin{equation}\label{Eq-C5}
\epsilon_{\pi'}((\rho,a-2k))=\epsilon_{\pi}((\rho,a)).
\end{equation}
\item If $\sigma$ is an irreducible  representation of a classical group such that
\begin{equation}\label{Eq-C5+}
\pi\h \nu^{(a-1)/2}\rho\t
\nu^{(a-3)/2}\rho\t\dots\t
\nu^{(a-2(k-1)-1)/2}\rho
\r \s
,
\end{equation}
then $\sigma\cong\pi'$. In particular, $\sigma$  is uniquely
determined by $\pi$ (and it is square integrable).

\end{enumerate}
\end{theorem}

\begin{proof} First we shall prove by induction that $\pi$ can be embedded into representation of type \eqref{Eq-C1-}. For $k=1$ this follows Lemma \ref{Le-sr}. Suppose that $[a-2(k+1),a-2]\cap \jrp=\emptyset$, and that we have an embedding 
$$
\pi\h \d([
\nu^{(a-2(k-1)-1)/2}\rho,
 \nu^{(a-1)/2}\rho])
\r \pi'.
$$
Now Lemma \ref{Le-sr}, \eqref{Eq-C11} and
assumption $[a-2(k+1),a-2]\cap \jrp=\emptyset$ implies that 
$$
\pi'\h \nu^{(a-2(k+1)-1)/2}\rho\r\pi''
$$
 for some square integrable representation $\pi'$ (use (1) of Proposition \ref{Pr-jb-main}). Further, (1) of Proposition \ref{Pr-jb-main} implies
$$
Jord(\pi'')=\left(Jord(\pi)\ \backslash\ \{(\rho,a)\}\right)\cup
\{(\rho,a-2(k+1))\}.
$$
 Observe that 
$$
\pi\h
 \d([
\nu^{(a-2(k-1)-1)/2}\rho,
 \nu^{(a-1)/2}\rho])
\t  \nu^{(a-2(k+1)-1)/2}\rho\r\pi''.
$$
We know 
$$
 \d([
\nu^{(a-2(k+1)-1)/2}\rho,
 \nu^{(a-1)/2}\rho])
\r\pi''
\h
 \d([
\nu^{(a-2(k-1)-1)/2}\rho,
 \nu^{(a-1)/2}\rho])
\t  \nu^{(a-2(k+1)-1)/2}\rho\r\pi''.
$$
Suppose
\begin{equation}
\label{Eq-C55}
\pi\not\h
 \d([
\nu^{(a-2(k+1)-1)/2}\rho,
 \nu^{(a-1)/2}\rho])
\r\pi''.
\end{equation}
Then  we can easily get 
$$
\pi\h
 \nu^{(a-2(k+1)-1)/2}\rho
\t
 \d([
\nu^{(a-2(k-1)-1)/2}\rho,
 \nu^{(a-1)/2}\rho])
\r\pi''.
$$
Now 
(1) of Proposition \ref{Pr-jb-jac} implies $a-2(k+1)\in \jrp$, which is a contradiction. This tells that \eqref{Eq-C55} can not happen. Therefore, we have proved the existence of an embedding of type \eqref{Eq-C1-}.

Observe that (2) follows directly from Proposition \ref{Pr-pcs}. Further, (3) follows from Proposition \ref{Pr-jb-jac}.

Let us now prove (1). Denote representation \eqref{Eq-C1} by $\Pi$.  To prove (1), it  is enough to prove that the multiplicity of 
\begin{equation}
\label{Eq-C6}
 \nu^{(a-1)/2}\rho\o
\nu^{(a-3)/2}\rho\o
\dots
\o
\nu^{(a-2(k-1)-1)/2}\rho
\o\pi'.
\end{equation}
in the Jacquet module of $\Pi$ is 1. For this, it is enough to prove that the multiplicity of 
$
\d([
\nu^{(a-2(k-1)-1)/2}\rho,
 \nu^{(a-1)/2}\rho])
\r \pi'
$
in $\mu^*(\Pi)$ is one. We shall now prove that if an irreducible representation
$$
\omega=
\d([
\nu^{(a-2(k-1)-1)/2}\rho,
 \nu^{(a-1)/2}\rho])
\r \s
$$
 is in $\mu^*(\Pi)$, then $\s\cong \pi'$ and it has multiplicity one in $\mu^*(\Pi)$.

We have
\begin{equation}
\label{Eq-prod}
\mu^*(\Pi)=\underset{i=(a-2(k-1)-1)/2}{\overset{(a-1)/2}{\t}}(1\o\nu^i\rho+
\nu^i\rho\o1+
\nu^{-i}\rho\o1)\r\mu^*(\pi').
\end{equation}
Suppose that for each index $i$ in the above formula we have picked up the term $\nu^i\rho\o\rho$. To get $\omega$ for a subquotient, we must take from $\mu^*(\pi')$ the term $1\o\pi'$. In the corresponding product, if $\omega$ is a subquotient, then $\s\cong\pi'$ and it has multiplicity one. 

Suppose that for some index $i$ in \eqref{Eq-prod} we have picked up one of two terms different from $\nu^i\rho\o 1$ and suppose that $\omega$ is a subquotient of the corresponding product. Obviously, the term $\nu^{-i}\rho$ can not give $\omega$ (we see this from the cuspidal supports).
Thus, the only remaining possibility is $\nu^i\rho\o1$. Now to get $\omega$ for a subquotient, we must  take  from $\mu^*(\pi')$  a non-zero term of form $\tau'\o\s$, where $\tau'$ is an irreducible representation of some $GL(l,F')$ with $l\geq 1$ such that the cuspidal support of  $\tau'$ is contained in
$\{\nu^{\frac{a-2(k-1)-1}2}\rho,\dots, \nu^{\frac{a-3}2}\rho,
\nu^{\frac{a-1}2}\rho\}$. This would imply that $\ell\in Jord_\rho(\pi')$ for some $\ell \in \{
a-2(k-1),\dots, a-2,
a\}$. This contradict to \eqref{Eq-C1--} and (3). This contradiction ends the proof of the claim.

We get (4) applying Lemma \ref{Le-red} several times.

Suppose that we have an embedding \eqref{Eq-C5+}. Then $\omega$ which we have defined above must be a
quotient of the Jacquet module of $\pi$, and therefore in $\mu^*(\Pi)$. We have seen that then $\s\cong\pi'$.
This ends the proof of (5).

The proof of the theorem is now complete.
\end{proof}

\begin{definition}\label{down} The representation $\pi'$ from the above theorem will be denoted by 
$$
\pi^{(\rho,a\downarrow a-2k)}.
$$
Further, the representation 
$\pi$ from the above theorem will be denoted by 
$$
(\pi')^{(\rho,a-2k\uparrow a)}.
$$
\end{definition}

\section{Appendix: an irreducibility result}\label{add}
\setcounter{equation}{0}

 The following simple, but often useful result  is proved in \cite{T-RPI} (Theorem 13.2). Let $\D$ be a segment in irreducible cuspidal representations of general linear groups and let $\sigma$ be an irreducible cuspidal representation of a classical group. Then, assuming that (BA) holds, we have
$$
\d(\D)\r\s \text{ reduces } \iff \rho\r\s \text{ reduces for some }\rho\in\D,
$$
or equivalently
$$
 \rho\r\s \text{ is irreducible for all }\rho\in\D \iff \d(\D)\r\s \text{ is irreducible}.
$$

If we take instead of cuspidal $\s$ an irreducible square integrable representation $\pi$, the above equivalence does not  hold in general (we can have $\d(\D)\r\pi $ irreducible despite the fact that $ \rho\r\pi$ reduces for some $\rho\in\D$).
Instead of the equivalence, for the square integrable representation  $\pi$ one implication  still holds. This follows from
G. Mui\'c's paper \cite{Mu3}, where he has described  completely (besides others) reducibility points of representations $\d(\D)\r\pi$.
Proof of this implication 
 is elementary in comparison with his results. For the convenience of the reader, we present here a proof of this implication (which  we have used  in this paper). Before we give the proof, we shall recall  (general) Proposition 6.1 from \cite{T-CJM}, which we shall use several times in the proof below (note that in  (vii) of Proposition 6.1 in \cite{T-CJM}, the condition weather $(\rho,2)$ satisfy or not satisfy (J1)  was forgotten).
 
 \begin{proposition}
 \label{Pr-red-si}
  Let $\rho$ be an irreducible $F'/F$-self dual cuspidal representation of a general linear group, and let
$\pi$ be an irreducible square integrable representation of  $S_q$.
Suppose that (BA) holds. 
Let
$a$ be a positive integer. Then:
\begin{enumerate}

\item [(i)] For $\alpha\in \mathbb R$, $\nu^\a\rho\r\pi$ reduces if and only
if
$\nu^{-\a}\rho\r\pi$ reduces.

\item[(ii)] If $\a\in \mathbb R \backslash (1/2)\mathbb Z$, then $\nu^\a\rho\r\pi$
is irreducible.

\item[(iii)] $\rho\rtimes \pi$ reduces if and only if $\rho$ has odd parity
with respect to $\pi$ and 
$1\not\in Jord_\rho(\pi)$.

\item [(iv)] If $a\not\in Jord_\rho(\pi)$, then
$\nu^{(a+1)/2}\rho\rtimes\pi$ is irreducible.

\item[(v)] If $a\in\jrp$ and $a+2\not\in\jrp$, then
$\nu^{(a+1)/2}\rho\rtimes\pi$  reduces.

\item[(vi)] Suppose that $a$ and $a+2$ are in $\jrp$. Then
$\nu^{(a+1)/2}\rho\rtimes\pi$  reduces if and only if $\e((\rho,a))\e((\rho,a+2))=1$.

\item[(vii)] $\nu^{1/2}\rho\r\pi$  reduces  if
and only if $(\rho,2)$ satisfies {\rm (J1)} and $2\not\in \jrp$, or $2\in\jrp$ and $\e((\rho,2))=1$. 

\noindent
In other words,
$\nu^{1/2}\rho\r\pi$ is irreducible if and only if $(\rho,2)$ does not satisfy {\rm (J1)}, or $2\in \jrp$ and $\e((\rho,2))=-1$.

\end{enumerate}
\end{proposition}

Recall, if  $\rho$ is a not $F'/F$-self dual irreducible  cuspidal representation of a general linear group and $\alpha\in\mathbb R$, then $\nu^\alpha\rho\r\pi$ is  irreducible ($\pi$ is an irreducible square integrable representation of  $S_q$).

\begin{lemma}
\label{Le-irr}
Let $\D$ be a segment in irreducible cuspidal representations of general linear groups.
Suppose that
$
\tau\r\pi
$
is irreducible for all $\tau\in \D$. Then
$$
\d(\D)\r\pi
$$
is irreducible.
\end{lemma}

Not that  for the above result, we assume also that the basic assumption (BA) from section \ref{notation}  holds.

\begin{proof} We shall prove the lemma by induction with respect to Card$(\D)$. For
Card$(\D)=1$ there is nothing to prove. Therefore we shall assume in what follows that Card$(\D)\geq 2$.

First consider the case when $\d(\D)$ is unitarizable. Write $\d(\D)=\d(\rho,a)$. Suppose that
$\d(\D)\r\pi$ is reducible. Then $(\rho,a)$ satisfies (J1) and $a\not\in\jrp$. Further suppose that  $\tau\r\pi$ is irreducible for all $\tau\in \D$.
Assume that $a$ is odd. 
Now (v) of Proposition \ref{Pr-red-si} (and the fact that $\nu^{(a-1)/2}\rho\r\pi , \nu^{(a-1)/2-1}\rho\r\pi,\dots,\nu\rho\r\pi,\nu\rho\r\pi$ are all irreducible by our assumptions) implies that  $a-2, a-4,\dots, 3,1$ are not
in
$\jrp$. But $1\notin\jrp$ implies that $\rho\r\pi$ reduces. Since $\rho\in \D$, we have obtained
a contradiction. It remains to consider the case of $a$  even. Now in the same
way as before (using (v) of Proposition \ref{Pr-red-si}), we get that $a-2,a-4,\dots,2$ are not in $\jrp$. Since $(\rho,2)$ satisfies (J1), (vii) of the same proposition implies that
$\nu^{1/2}\rho\r\pi$ reduces. This is again a contradiction. This completes the proof of the lemma in the case of $\d(\D)$ unitarizable.

Now we go to the case when $\d(\D)$ is not unitarizable.  We shall first consider a delicate case,   the case  of $\D=\{\rho,\nu\rho\}$, where $\rho$ is an irreducible $F'/F$-self dual cuspidal representation of a general linear group such that $\rho\r\pi$ and $\nu\rho\r\pi$ are both irreducible. Note that a representation of
the form
$\nu^\alpha\rho\o\tau$, $\a\leq 0$, cannot be a subquotient of the Jacquet module of $\pi$, since $\pi$ is
square integrable (this follows directly from the square integrability criterion of Casselman from \cite{C-int}). We consider now  two possibilities.  Suppose first that the parity of $\rho$ is odd (i.e. that $(\rho,1)$ satisfies (J1)). Then first $1\in\jrp$ (because $\rho\r\pi$ is irreducible). Now (v) of Proposition \ref{Pr-red-si} implies $3\in\jrp$ (because 
$\nu\rho\r\pi$ is irreducible). Further, (vi) of the same proposition and the irreducibility of $\nu\rho\r\pi$ imply
$\epsilon_\pi((\rho,1))\epsilon_\pi((\rho,3))^{-1}=1$. This and the definition of $\epsilon_\pi$ (using Proposition \ref{Pr-eps=}) imply that a
representation of the form
$\nu\rho\o\tau$ can not be  a subquotient of the Jacquet module of $\pi$. 
If the parity
of
$\rho$ is even, then a representation of the form $\nu\rho\o\tau$ 
 can not again be a subquotient of  of the Jacquet module of $\pi$ by (1) of Proposition \ref{Pr-jb-jac} (since 3 is odd).

By  the proof of the lemma for $\d(\D)$ unitarizable, our assumptions on $\rho$ and $\pi$ imply that $\d([\nu^{-1}\rho,\nu\rho])\r\pi$ is
irreducible. Now \cite{G} (or \cite{Moe-T}) implies that  $\rho\t \d([\nu^{-1}\rho,\nu\rho])\r\pi$ is irreducible.
Recall
$$
\mu^*(\rho \times \rho \times \nu \rho \times \nu \rho
\rtimes \pi) =
(1\o\rho+2\rho\o1)^2
\t
(1\o\nu\rho+\nu\rho\o1+\nu^{-1}\rho\o1)^2
\r
\mu^*(\pi).
$$
From this formula and from the above observations about Jacquet modules of $\pi$, one gets directly that the
multiplicity of $
\delta([\rho,\nu \rho]) \times \delta([\rho, \nu \rho])
\otimes \pi $
in
$
\mu^*(\rho \times \rho \times \nu \rho \times \nu \rho
\rtimes \pi) $ is four.
Further, one gets easily that 
the  multiplicity of $\delta([\rho,\nu \rho]) \times \delta([\rho, \nu \rho])
\otimes \pi$  in
$\mu^*(\rho \times
\delta([\nu^{-1}\rho,\nu\rho]) \rtimes \pi)$ is also four.

 Let $\theta_s$ be an
irreducible subrepresentation of $\delta([\rho, \nu \rho]) \rtimes \pi$. Note that
$\delta([\rho, \nu \rho]) \o \pi\leq \mu^*(\theta_s)$ (by Frobenius reciprocity). This implies 
\begin{equation}
\label{Eq-app-1}
\delta([\rho, \nu
\rho])
\t \delta([\rho, \nu \rho]) \o \pi\leq  \mu^*(\delta([\rho, \nu \rho]) \rtimes \theta_s).
\end{equation}
Since the
 multiplicity
 of $\delta([\rho,\nu \rho]) \times \delta([\rho, \nu \rho])
\otimes \pi$ in the Jacquet modules of  $\rho \times \rho \times \nu \rho \times \nu \rho
\rtimes \pi$ and $\rho \times
\delta([\nu^{-1}\rho,\nu\rho]) \rtimes \pi$ is four in both cases, and the last representation is irreducible, we see that
\begin{equation}
\label{Eq-app-2}
\rho \times
\delta([\nu^{-1}\rho,\nu\rho]) \rtimes \pi\leq \delta([\rho, \nu \rho]) \rtimes \theta_s.
\end{equation}

From the Langlands classification follows that the representation $\delta([\rho, \nu \rho]) \rtimes \pi$ has a unique irreducible quotient. Denote it by $\theta_q$. Then $\theta_q\h \delta([\nu^{-1}\rho, \rho ]) \rtimes \pi$, which implies $ \delta([\nu^{-1}\rho, \rho ]) \o \pi \leq \mu^*(\theta_q)$.

Suppose that $\delta([\rho, \nu \rho]) \rtimes \pi$ reduces. Then $\theta_s\not\cong\theta_q$.
Now starting with \eqref{Eq-app-2}, we get
$$
\aligned
\mu^*(\rho \times
\delta([\nu^{-1}\rho,\nu\rho]) \rtimes \pi)\leq \mu^*(\delta([\rho, \nu \rho]) \rtimes
\theta_s)
\hskip60mm
\\
\hskip40mm
\leq M^*(\delta([\rho,\nu\rho]))\r
\bigg(M^*(\delta([\rho,\nu\rho]))\r\mu^*(\pi)-\delta([\nu^{-1}\rho,\rho])\o
\pi\bigg).
\endaligned
$$
One directly gets that the multiplicity of $\rho \times
\delta([\nu^{-1}\rho,\nu\rho]) \o \pi$ in $\mu^*(\rho \times
\delta([\nu^{-1}\rho,\nu\rho]) \rtimes \pi)$ is  four (for the proof we need only that it is at least four).
This and the above remarks about Jacquet modules of $\pi$ (regarding the terms of the form $\nu\rho\o\tau$ and $\nu^\a\rho\o\tau$ for $\a\leq0$) imply
$$
\aligned
4\rho \times \delta([\nu^{-1}\rho, \nu\rho]) \otimes \pi
\leq \big(\delta([\rho, \nu\rho]) + \rho\t\nu\rho+
\delta([\nu^{-1}\rho, \rho])\big)\t \big(\delta([\rho, \nu\rho]) +
\rho\t\nu\rho\big)\o\pi.
\endaligned
$$
This implies
$$
4\rho \times \delta([\nu^{-1}\rho, \nu\rho]) \otimes \pi 
\leq 
\delta([\nu^{-1}\rho, \rho])\t \big(\delta([\rho, \nu\rho]) +
\rho\t\nu\rho\big)\o\pi.
$$
Obviously, this cannot happen. Therefore, we have proved the irreducibility of
$\delta([\rho, \nu\rho]) \r \pi $.

We consider now the case $\D=[\nu^\a \rho , \nu^\b\rho]$, where $\a,\b\in \mathbb R$, $\b-\a\in \mathbb Z_{>0}$ and $\rho$ is an $F'/F$-selfdual irreducible cuspidal representation of a general linear group. Suppose that $\D$ satisfies the condition of the lemma (with respect to $\pi$). Since $\d([\nu^\a \rho , \nu^\b\rho])\r\pi$ and $\d([\nu^{-\b} \rho , \nu^{-\a}\rho])\r\pi$ have the same composition series, it is enough to prove the irreducibility of $\d(\D)\r\pi$ in the case  $\a+\b> 0$. 
By the previous part of the proof, it is enough to consider the case $\D\ne \{\rho,\nu\rho\}$ (therefore if  $\a+1=\b$, then  $\a\ne 0$), which we shall assume in what follows.

Let $\theta_s$ be an irreducible subrepresentation of $\d(\D)\r\pi$. Note that $\d(\D)\r\pi$
has a unique irreducible quotient. Now a well-known embedding  of $\d(\D)$ and the
inductive assumption imply
$$
\theta_s \hookrightarrow \d(\D)\r\pi
\h \nu^\b\rho\t\d([\nu^{\a}\rho,\nu^{\b-1}\rho])\r\pi
\cong \nu^\b\rho\t\d([\nu^{-\b+1}\rho,\nu^{-\a}\rho])\r\pi.
$$
Observe $\b+1
\ne-\b+1$ (since $\b>0$).

Suppose $-\a+1\ne\b$, i.e. $\a\ne-\b+1$. Then 
$\nu^\b\rho\t\d([\nu^{-\b+1}\rho,\nu^{-\a}\rho])$ is irreducible  (recall $\b+1
\ne-\b+1$), which implies
$$
\theta_s \hookrightarrow  
 \d([\nu^{-\b+1}\rho,\nu^{-\a}\rho])\t \nu^\b\rho\r\pi
\cong \d([\nu^{-\b+1}\rho,\nu^{-\a}\rho])\t \nu^{-\b}\rho\r\pi.
$$
Now Frobenius reciprocity implies that  $\d([\nu^{-\b+1}\rho,\nu^{-\a}\rho])\otimes
\nu^{-\b}\rho\otimes\pi$ is a quotient of the Jacquet module of $\theta_s$. This and the transitivity of Jacquet modules imply that 
$\d([\nu^{-\b}\rho,\nu^{-\a}\rho])\otimes\pi$ must be also a subquotient of the Jacquet module
of $\theta_s$. But then $\theta_s$ must be the unique irreducible quotient of $\d(\D)\r\pi$
(see Lemma 4.4 of \cite{Moe-T}). This implies the irreducibility of $\d(\D)\r\pi$. 

It remains to consider the case $\a=-\b+1$. Note that $\b-\a\in \mathbb Z_{>0}$ implies $\b\geq 1$, and further $\D\ne \{\rho,\nu\rho\}$  implies $b>1$. For an irreducible subrepresentation $\theta_s$ of $\d([\nu^{-\b+1}\rho,\nu^\b\rho])\r\pi$, we proceed similarly as above:
$$
\theta_s\h
\d([\nu^{-\b+1}\rho,\nu^\b\rho])\r\pi\h
\d([\nu^{-\b+2}\rho,\nu^\b\rho])\t \nu^{-\b+1}\rho \r\pi
\cong \d([\nu^{-\b+2}\rho,\nu^\b\rho])\t \nu^{\b-1}\rho \r\pi.
$$
Clearly, $\b+1\ne\b-1$. Further, $\b-1+1=-\b+2$ implies $\b=1$, which contradicts to $b>1$. Therefore, $\d([\nu^{-\b+2}\rho,\nu^\b\rho])\t \nu^{\b-1}\rho$ is
irreducible. Using the inductive assumption, we continue similarly as in the previous case:
$$
\theta_s\h
 \nu^{\b-1}\rho \t \d([\nu^{-\b+2}\rho,\nu^\b\rho]) \r\pi
\cong 
 \nu^{\b-1}\rho \t \d([\nu^{-\b}\rho,\nu^{\b-2}\rho]) \r\pi.
$$
This implies that $\nu^{\b-1}\rho\o\nu^{\b-2}\rho\o\dots\o\nu^{-\b}\rho\o\pi$ is in
the Jacquet module of $\s$. The transitivity of Jacquet modules imply that
$\d([\nu^{-\b}\rho,\nu^{\b-1}\rho])\o\pi$ is also in the Jacquet module of $\theta_s$. One
concludes the irreducibility of
$\d(\D)\r\pi$ in the same way as above, using Lemma 4.4 of \cite{Moe-T}.

We end with the case  $\D=[\nu^\a \rho , \nu^\b\rho]$, where $\a,\b\in \mathbb R$, $\b-\a\in \mathbb Z_{>0}$ and $\rho$ is a unitarizable irreducible cuspidal representation of a general linear group which is not  $F'/F$-selfdual (then $\D$ satisfies the condition of the lemma). It is enough to prove the irreducibility of $\d(\D)\r\pi$ in the case  $\a+\b> 0$.  Let $\theta_s$ be an irreducible subrepresentation of $\d(\D)\r\pi$. Then
$$
\theta_s \hookrightarrow \d(\D)\r\pi
\h \nu^\b\rho\t\d([\nu^{\a}\rho,\nu^{\b-1}\rho])\r\pi
\cong \nu^\b\rho\t\d([\nu^{-\b+1}\check\rho,\nu^{-\a}\check\rho])\r\pi
$$
$$
\hskip40mm
\cong \d([\nu^{-\b+1}\check\rho,\nu^{-\a}\check\rho])\t\nu^\b\rho\r\pi
\cong \d([\nu^{-\b+1}\check\rho,\nu^{-\a}\check\rho])\t\nu^{-\b}\check\rho\r\pi.
$$
We conclude now the irreducibility of $\d(\D)$ in the same way as in the previous two cases.

The proof of the
lemma is now complete.
\end{proof}

\end{document}